\documentclass[10pt, draft]{amsart}
\usepackage{amssymb, amstext, amscd, amsmath, amssymb}
\usepackage{mathtools, xypic, paralist, color, dsfont, rotating}
\usepackage{verbatim}
\usepackage{enumerate,enumitem}
\usepackage{float}
\usepackage{setspace}
\usepackage{eso-pic}% http://ctan.org/pkg/eso-pic
\usepackage{lipsum}% http://ctan.org/pkg/lipsum
\usepackage{tikz}

\floatstyle{boxed} 
\restylefloat{figure}

\numberwithin{equation}{section}
\usepackage[a4paper]{geometry}
\usepackage{quoting}
\quotingsetup{vskip=.1in}
\quotingsetup{leftmargin=.17in}
\quotingsetup{rightmargin=.17in}
\let\OLDthebibliography\thebibliography
\renewcommand\thebibliography[1]{
  \OLDthebibliography{#1}
  \setlength{\parskip}{0pt}
  \setlength{\itemsep}{2pt plus 0.5ex}
}

\makeatletter
\def\@cite#1#2{{\m@th\upshape\bfseries%
[{#1\if@tempswa{\m@th\upshape\mdseries, #2}\fi}]}}
\makeatother
\theoremstyle{plain}
\newtheorem{theorem}{Theorem}[section]
\newtheorem{corollary}[theorem]{Corollary}
\newtheorem{proposition}[theorem]{Proposition}
\newtheorem{lemma}[theorem]{Lemma}
\theoremstyle{definition}
\newtheorem{definition}[theorem]{Definition}
\newtheorem{example}[theorem]{Example}

\newtheorem{remark}[theorem]{Remark}

\newtheorem*{acknow}{Acknowledgements}

\theoremstyle{remark}
\mathtoolsset{centercolon}
  \newcommand{\A}{{\mathcal{A}}}
  \newcommand{\B}{{\mathcal{B}}}
  \newcommand{\C}{{\mathcal{C}}}

  \newcommand{\F}{{\mathcal{F}}}
  
\renewcommand{\H}{{\mathcal{H}}}
  \newcommand{\I}{{\mathcal{I}}}
  
  \newcommand{\K}{{\mathcal{K}}}
\renewcommand{\L}{{\mathcal{L}}}
  
  \newcommand{\N}{{\mathcal{N}}}
\renewcommand{\O}{{\mathcal{O}}}

\renewcommand{\S}{{\mathcal{S}}}
  \newcommand{\T}{{\mathcal{T}}}

  \newcommand{\X}{{\mathcal{X}}}
  \newcommand{\Y}{{\mathcal{Y}}}
  
\def\al{\alpha}
\def\be{\beta}

\def\ga{\gamma}
\def\De{\Delta}
\def\de{\delta}

\def\la{\lambda}

\newcommand{\bC}{\mathbb{C}}

\newcommand{\bZ}{\mathbb{Z}}

\newcommand{\fA}{{\mathcal{A}}}

\newcommand{\fG}{{\mathfrak{G}}}
\newcommand{\fg}{{\mathfrak{g}}}

\newcommand{\fh}{{\mathfrak{h}}}

\newcommand{\foral}{\text{ for all }}
\newcommand{\qand}{\quad\text{and}\quad}

\newcommand{\qiff}{\quad\text{if and only if}\quad}
\newcommand{\qfor}{\quad\text{for}\quad}
\newcommand{\qforal}{\quad\text{for all}\ }

\newcommand{\ca}{\mathrm{C}^*}

\newcommand{\cenv}{\mathrm{C}^*_{\textup{env}}}

\newcommand{\ol}{\overline}
\newcommand{\wt}{\widetilde}
\newcommand{\wh}{\widehat}
\newcommand{\sot}{\textsc{sot}}

\newcommand{\ad}{\operatorname{ad}}

\newcommand{\Aut}{\operatorname{Aut}}

\newcommand{\env}{\operatorname{\textup{env}}}

\newcommand{\id}{{\operatorname{id}}}

\newcommand{\mt}{\emptyset}

\newcommand{\spn}{\operatorname{span}}

\newcommand{\sca}[1]{\left\langle#1\right\rangle} %\sca{a,b} =<a,b>
\newcommand{\nor}[1]{\left\Vert #1\right\Vert} %\nor{x}=||x||
\newcommand{\quo}[2]{{\raisebox{.1em}{$#1$}\left/ \, \raisebox{-.1em}{$#2$}\right.}} % quotient of objects
\newtheorem*{theorem*}{Theorem}
\newtheorem*{corollary*}{Corollary}
\newtheorem*{proposition*}{Proposition}
\newtheorem*{lemma*}{Lemma}
\newtheorem*{remark*}{Remark}
\newtheorem*{definition*}{Definition}

\usepackage[normalem]{ulem}
\usepackage{xstring}

\def\lcm{\operatorname{lcm}}
\def\dep{\ol{\de}^+}

\def\couni{C}

\def\iotenv{\iota_{\textup{env}}}
\def\delenv{\delta_{\textup{env}}}

\def\redcov{\T_\la(X)/ q_\la(\I_\infty)}

\begin{document}
%%%%%%%%%%%%%%%%%%%%%%%%%%%%%%%%
%\AddToShipoutPictureBG*{%
%  \AtPageUpperLeft{%
%    \hspace{\paperwidth}%
%    \raisebox{-\baselineskip}{%
%      \makebox[0pt][r]{All corrections appear in blue. Numbers in front of them refer to the corresponding comment by the referee}
%}}}%
%%%%%%%%%%%%%%%%%%%%%%%%%%%%%%%%
\title[C*-envelopes for coactions on operator algebras]
{C*-envelopes for operator algebras with a coaction and co-universal C*-algebras for product systems}
\date{22 December 2020, modified 19 January  2022}
%%%%%%%%%%%%%%%%%%%%%%%%%%%%%%%%

\author[A. Dor-On]{A. Dor-On}
\address{Department of Mathematical Sciences\\ University of Copenhagen \\ Copenhagen \\ Denmark}
\email{adoron@math.ku.dk}

\author[E.T.A. Kakariadis]{E.T.A. Kakariadis}
\address{School of Mathematics, Statistics and Physics\\ Newcastle University\\ Newcastle upon Tyne\\ NE1 7RU\\ UK}
\email{evgenios.kakariadis@newcastle.ac.uk}

\author[E.G. Katsoulis]{E. Katsoulis}
\address{Department of Mathematics\\ East Carolina University\\ Greenville\\ NC 27858\\USA}
\email{katsoulise@ecu.edu}

\author[M. Laca]{M. Laca}
\address{Department of Mathematics and Statistics\\ University of Victoria\\ Victoria\\ BC\\ Canada V8W 2Y2}
\email{laca@uvic.ca}

\author[X. Li]{X. Li}
\address{School of Mathematics and Statistics\\ University of Glasgow\\ University Place\\ Glasgow\\ G12 8QQ\\ UK}
\email{xin.li@glasgow.ac.uk}



%%%%%%%%%%%%%%%%%%%%%%%%%%%%%%%%
\begin{abstract}
A cosystem consists of a possibly nonselfadoint operator algebra equipped with a coaction by a discrete group.  We introduce the concept of  C*-envelope for a cosystem; roughly speaking, this is the smallest C*-algebraic cosystem  that contains an equivariant completely isometric copy of the original one.
We show that the C*-envelope for a cosystem always exists and we explain how it relates to the usual C*-envelope.
We then show that for compactly aligned product systems over group-embeddable right LCM semigroups,
the C*-envelope is co-universal, in the sense  of Carlsen, Larsen, Sims and Vittadello, 
for the Fock tensor algebra equipped with its natural coaction.
This yields 
 the existence of  a co-universal C*-algebra, generalizing previous results of Carlsen, Larsen, Sims and Vittadello, and of  Dor-On and Katsoulis. We also realize the C*-envelope of the tensor algebra as the reduced cross sectional algebra of a Fell bundle introduced by Sehnem, which, under a mild assumption of normality, we then identify with the quotient of the Fock algebra by the image of Sehnem's strong covariance ideal. In another application, we  obtain a reduced Hao-Ng isomorphism theorem for the co-universal algebras.
\end{abstract}

\maketitle

%\newpage

%\doublespacing
%\tableofcontents
%\singlespacing

%\newpage

%%%%%%%%%%%%%%%%%%%%%%%%%%%%%%%%
\section{Introduction}
%%%%%%%%%%%%%%%%%%%%%%%%%%%%%%%%

In a remarkable series of papers  spanning four decades, Arveson developed a non-commutative analogue of boundary theory for nonselfadjoint operator algebras \cite{Arv69, Arv72, Arv98, Arv08}, which  constitutes one of the most fundamental and fruitful areas of interaction between C*-algebras and nonselfadjoint operator algebras. One specific noncommutative boundary is the noncommutative analogue of the Shilov boundary, called the C*-envelope, which can be thought of as the smallest C*-algebra containing the given operator algebra in a reasonable sense. Computing the C*-envelope in various cases has been of interest and use to many authors over the years \cite{DFK17, DK11, DOM16, DS18, Hum20, KR16,  MS98}. C*-envelopes have also had recent applications  in classification of nonselfadjoint operator algebras \cite{DEG20}, finite dimensional approximation \cite{CR19}, crossed products \cite{KR16},
group theory \cite{BKKO17,KK17}, noncommutative geometry \cite{CvS+}, and noncommutative convexity \cite{EH19}.

In the seminal work of Pimsner \cite{Pim95}, many operator algebra constructions were generalized and unified by associating them to a single C*-correspondence. 
This allowed Pimsner to generalize the work of Cuntz and Krieger \cite{CK80}, as well as many others. Pimsner's work was refined by Katsura~\cite{Kat04}, who removed all conditions on the C*-correspondence  to obtain an in-depth study of what we now call Cuntz-Pimsner algebras.  A natural context for further generalization, unification  and insight was introduced by Fowler in his work on discrete product systems of C*-correspondences over quasi-lattice ordered semigroups \cite{Fow02}. 
Fowler's Toeplitz-Nica-Pimsner algebras generalize Nica's Wiener-Hopf algebras from \cite{Nic92}, as well as Pimsner's Toeplitz algebras. Although Fowler provided a Cuntz-type algebra for regular product systems when the semigroup is directed, for many years it was unclear what the right notion of a Cuntz algebra of a product system should be. 
Eventually, Sims and Yeend \cite{SY10} were able to give a definition of a Cuntz-Nica-Pimsner algebra $\N\O(X)$ for many new product systems. In further work, Carlsen, Larsen, Sims and Vitadello \cite{CLSV11} introduced  a co-universal Cuntz-Nica-Pimsner type of algebra, which they denoted as $\N\O^r(X)$, 
that satisfies an appropriate uniqueness theorem for equivariant homomorphisms. Their co-universal algebra  was shown to exist under additional hypothesis on the  product system; it generalizes  reduced  crossed products by  quasi-lattice ordered groups,
Crisp-Laca boundary quotients \cite{CL07},  and higher-rank graph C*-algebras \cite{KuP}.  

The tensor algebra $\mathcal{T}_{\la}(X)^+$ is the canonical nonselfadjoint subalgebra of the reduced Toeplitz C*-algebra, or Fock C*-algebra, $\mathcal{T}_{\la}(X)$ generated by the left-creation operators of the C*-correspondences that comprise the product system. It models many of the nonselfadjoint operator algebras that were previously investigated in the multivariable setting \cite{DK11,KK12}.
As with any nonselfadjoint 
algebra, a fundamental problem regarding $\mathcal{T}_{\la}(X)^+$ is the identification of its C*-envelope $\cenv(\mathcal{T}_{\la}(X)^+)$. In the case of a single $\ca$-correspondence this was done by Katsoulis and Kribs \cite{KatK}, following earlier work of Muhly and Solel \cite{MS98}, who pioneered the study of tensor algebras. In \cite{DFK17}, Davidson, Fuller and Kakariadis identified the C*-envelope for tensor algebras of product systems associated with $\bZ^n$-dynamical systems. In that paper dilation theoretic techniques merged with  uniqueness theorems for the images of  equivariant homomorphisms and gave strong motivating evidence that one can use C*-envelope techniques in order to prove the existence of a co-universal object for more general product systems over abelian orders.
This approach was fully materialized by Dor-On and Katsoulis \cite{DK20} who proved that for a compactly aligned product system $X$ over any abelian lattice ordered semigroup, $\cenv(\mathcal{T}_{\la}(X)^+)$ has the  co-universal property proposed in \cite{CLSV11},  thus showing in particular that  the co-universal algebra $\N\O^r(X)$ of \cite{CLSV11} exists without the injectivity assumption when the semigroup is abelian lattice ordered. This result strengthened the important connection between nonselfadjoint and selfadjoint operator algebra theory and raised the tantalising possibility  of proving directly the existence of an appropriate notion of C*-envelope that satisfies the desired co-universal property automatically  beyond abelian orders. Even though some of the techniques of \cite{DK20} are indeed applicable to more general settings,  it soon became clear
 that significant progress would require new ideas. The purpose of the present paper is to realize this possibility through the use of an equivariant version of the C*-envelope.

The turning point in our investigation is the realization that if $X$ is a product system over 
a subsemigroup $P$ of a group $G$, then the tensor algebra $\mathcal{T}_{\la}(X)^+$ comes equipped with a  natural (normal) coaction of $G$, forming what we call a cosystem. A cosystem $(\A, G, \de)$ consists of an operator algebra $\A$ and a coaction $\de \colon \A \to \A \otimes \ca(G)$ by a discrete group $G$. In Section~\ref{S;cosystem} we develop a boundary theory for cosystems that parallels the corresponding theory for operator algebras. In particular, given a cosystem $(\A, G, \de)$, we define notions of  $\ca$-cover and C*-envelope for $(\A, G, \de)$. Both  notions are equivariant analogues of the classical definitions.  In Theorem~\ref{T:co-env}, which is the main result of Section~\ref{S;cosystem}, we show that the C*-envelope $\cenv(\A, G, \de)$ of a cosystem $(\A, G, \de)$ always exists. Furthermore we give a picture of $\cenv(\A, G, \de)$ that connects it with the C*-envelope of $\A$. Specifically, $\cenv(\A, G, \de)$ is the $\ca$-subalgebra of $\cenv(\A)\otimes \ca(G)$ generated by $\de(\A)$, equipped with the coaction $\id\otimes \Delta$, where $\Delta$ is the comultiplication on $G$.

Having developed a satisfactory theory of C*-envelopes for cosystems, we  then move to applications. In Section~\ref{S;main} we investigate various C*-algebras associated with a product system over a right LCM subsemigroup of a group.  Right LCM semigroups that embed in a group   include quasi-lattice orders   as well as several other important  classes of semigroups. In Theorem~\ref{T:co-univ} we show that for every compactly aligned product system $X$ over a right LCM subsemigroup $P$ of a group $G$, the C*-envelope of the cosystem $(\T_\la(X)^+, G, \dep) $  has the co-universal property with respect to injective, gauge equivariant representations of $X$.  This resolves the problem of existence of a co-universal C*-algebra,
which is one of the central problems raised by  Carlsen, Larsen, Sims and Vitadello in \cite{CLSV11}. 
Specifically, our Theorem~\ref{T:co-univ} removes all the injectivity assumptions on $X$  from \cite[Theorem 4.1]{CLSV11} and at the same time generalizes the results for abelian semigroups of \cite{DK20} to the realm of right LCM semigroups that embed in a group. We remark that all the necessary facts from the theory of product systems over these semigroups are developed here from scratch, so, in particular, the proof of Theorem~\ref{T:co-univ} is essentially self-contained as it requires only some additional basic facts regarding the cross sectional C*-algebra of a Fell bundle \cite{Exe97}. 

In \cite{Seh18} Sehnem introduced a covariance C*-algebra $A\times_X P$ associated to a product system $X$ over a general subsemigroup $P$ of a group $G$ with coefficients in a C*-algebra $A$. There is a natural coaction of $G$ on $A\times_X P$ giving a grading 
$\S\C X := \{ [A \times_X P]_g \}_{g \in G}$.
Sehnem's covariance algebra satisfies an important property: any representation of $A\times_X P$ that is injective on $A$ is automatically injective on the fiber $[A \times_X P]_e$ over the identity. In Theorem~\ref{T:co-un is Fell} we show that if the product system $X$ is compactly aligned over a right LCM subsemigroup of a group, then our C*-envelope  is naturally isomorphic to the reduced cross-sectional algebra of $\S\C X$, while Sehnem's algebra $A \times_X P$ is isomorphic to the full cross-sectional algebra of  $\S\C X$.  We also consider the quotient  of $\T_\la(X) $ by the image  of Sehnem's ideal $\I_\infty$,  and in Corollary~\ref{C:exa Seh} we show that under a mild assumption (which is satisfied by all right LCM subsemigroups of exact groups), our C*-envelope is canonically isomorphic to this quotient: $ \cenv(\T_\la(X)^+, G, \dep)  \simeq \redcov$. When combined with our main result, Theorem~\ref{T:co-univ}, these results give a very detailed picture of $\cenv(\T_\la(X)^+, G, \dep) $.

In the final section of the paper we give an application of our theory to Hao-Ng type isomorphisms, much in the spirit of earlier works \cite{DK20, Kat17, KR16}.

%%%%%%%%%%%%%%%%%%%%%%%%%%%%%%%%
\section{Preliminaries}
%%%%%%%%%%%%%%%%%%%%%%%%%%%%%%%%
If $\X$ and $\Y$ are subspaces of some $\B(H)$ then we write $\ol{\X \Y}:= \ol{\spn}\{x y \mid x \in \X, y \in \Y\}$.
We denote the spatial tensor product by $\otimes$. 
All the semigroups considered in this paper are assumed to embed in a group and to contain the identity.

%%%%%%%%%%%%%%%%%%%%%%%%%%%%%%%%
\subsection{Operator algebras}
%%%%%%%%%%%%%%%%%%%%%%%%%%%%%%%%

We begin by  establishing terminology and recalling some fundamental facts in the theory of operator algebras.
For additional details and proofs, the monographs \cite{BL04, Pau02} provide an excellent introduction to the subject.

By an \emph{operator algebra} $\A$ we mean a norm-closed subalgebra of some $\B(H)$ for a Hilbert space $H$.
By a \emph{representation} of $\A$ we mean a completely contractive homomorphism $\phi \colon \A \to \B(H)$.
When $\phi \colon \A \to \B(H)$ is a representation, we will always assume it is non-degenerate in the sense that $\phi(\A)H$ is dense in $H$.

Meyer \cite{Mey01} has established the passage from the unital to the non-unital theory, which we now explain.
Suppose that $\A \subseteq \B(H)$ and $I_H \notin \A$.
Meyer shows that if $\phi \colon \A \to \B(K)$ is a (completely isometric) representation then the extension $\phi^1 \colon \A^1 \to \B(K)$ given by 
\[
\phi^1(a + \la I_H) = \phi(a) + \la I_K \qfor \A^1 = \spn\{\A, I_H\}
\]
is also a (completely isometric resp.) representation.
Hence $\A^1$ is independent of the completely isometric representation of $\A$ and thus constitutes \emph{the unique} ``one-point" unitization of $\A$.

A \emph{dilation} of a representation $\phi \colon \A \to \B(H)$ is a representation $\phi' \colon \A \to \B(H')$ such that $H \subseteq H'$ and $\phi(a) = P_H \phi'(a) |_H$ for all $a \in \A$.
A representation $\phi \colon \A \to \B(H)$ is called \emph{maximal} if every dilation $\phi' \colon \A \to \B(H')$ of $\phi$ is trivial, in the sense that $H$ is reducing for $\phi(\A)$.
The existence of maximal dilations in the unital case was first established by Dritschel and McCullough \cite{DM05}, and later simplified by Arveson \cite{Arv08}.
Dor-On and Salomon \cite{DS18} have shown that a representation $\phi$ is maximal if and only if its unitization $\phi^1$ is so. Hence, maximal dilations exist for possibly non-unital operator algebras, as arising from maximal dilations of their unitizations.

Now consider $\A$ inside the C*-algebra $\ca(\A)$ it generates.
By passing to the unitization and applying Arveson's Extension Theorem we see that every representation $\phi \colon \A \to \B(H)$ admits an extension $\wt{\phi} \colon \ca(\A) \to \B(H)$ to a completely contractive and completely positive map (ccp).
A representation $\phi \colon \A \to \B(H)$ is said to have the \emph{unique extension property (UEP)} if every ccp  extension to $\ca(\A)$ is a $*$-representation, and thus $\phi$ has a unique extension to a $*$-representation of $\ca(\A)$.
Arveson \cite{Arv08} shows that a representation is maximal if and only if it has the UEP in the unital case.
Dor-On and Salomon \cite{DS18} have extended this to the non-unital case as well.

The existence of maximal dilations leads naturally to the concept of the C*-envelope for an operator algebra, which we now discuss.
We say that $(C, \iota)$ is a \emph{C*-cover} of $\A$ if $\iota \colon \A \to C$ is a completely isometric representation with $C = \ca(\iota(\A))$.
The \emph{C*-envelope} $\cenv(\A)$ of $\A$ is a C*-cover $(\cenv(\A), \iota)$ with the following universal property:
if $(C', \iota')$ is a C*-cover of $\A$ then there exists a (necessarily unique) $*$-epimorphism $\phi \colon C' \to \cenv(\A)$ such that the following diagram
\begin{equation} \label{eq;env}
\xymatrix{
& & C' \ar@{.>}[d]^{\phi} \\
\A \ar[rru]^{\iota'} \ar[rr]^{\iota} & & \cenv(\A)
}
\end{equation}
commutes.
Arveson predicted the existence of the C*-envelope, which he computed for a variety of operator algebras \cite{Arv69}, but the existence problem was open for a decade until
 Hamana  solved it for unital algebras by proving the existence of injective envelopes \cite{Ham79}.
It follows from \cite[Subsection 2.2]{DS18} that the C*-envelope of an operator algebra $\A$ is the C*-algebra generated by a maximal completely isometric representation, even when $\A$ is non-unital.

\subsection{C*-correspondences}

A \emph{$\ca$-correspondence} $X$ over $A$ is a right Hilbert module over a C*-algebra $A$ with a left action of $A$
given by a $*$-homomorphism $\varphi_X$ of $ A $ into  the adjointable operators on $X$. 
We write $\L X$ for the adjointable operators on $X$ and $\K X$ for the norm closure of the generalized finite rank operators on $X$.
For two C*-corresponden\-ces $X, Y$ over the same $A$ we write $X \otimes_A Y$ for the balanced tensor product over $A$.
We say that $X$ is unitarily equivalent to $Y$ (symb. $X \simeq Y$) if there is a surjective adjointable operator $U \in \L(X,Y)$ such that $\sca{U \xi, U \eta} = \sca{\xi, \eta}$ and $U (a \xi b) = a U(\xi) b$ for all $\xi, \eta \in X$ and $a,b \in A$.

A \emph{Toeplitz representation} $(\pi,t)$ of a C*-correspondence $X$ over $A$ is a pair $(\pi,t)$ such that $\pi : A \rightarrow B(\H)$ is a $*$-representation and $t$ is a left module map implemented by $\pi$ which satisfies $\pi (\langle \xi , \eta \rangle) = t(\xi)^* t(\eta)$. Then $t$ is automatically a bimodule map via $\pi$.  When $(\pi,t)$ is a Toeplitz representation of the C*-correspondence $X$, there exists an induced $*$-representation of $\K X$ denoted by $t$ as well and  determined by $t(\theta_{\xi, \eta}) = t(\xi) t(\eta)^*$ for all rank-one operators $\theta_{\xi, \eta} \in \K X$.

\subsection{Product systems and their representations}

Let $P$ be a unital discrete subsemigroup of a discrete group $G$.
We will write $P^* = P\cap P^{-1}$ for the set of invertible elements in $P$.
We write $\ca_\la(P)$ for the C*-algebra generated by the left regular representation of $P$, i.e., the shift operators $V_p$ on $\ell^2(P)$ given by
\[
V_p e_r = e_{pr} \foral r \in P,
\]
where $\{e_r\}_{r\in P}$ is the standard orthonormal basis for $\ell^2(P)$.
\begin{definition} 
A \emph{product system $X$ over $P$ with coefficients in a C*-algebra $A$} is a family $\{X_p \mid p \in P\}$ of C*-correspondences over $A$ together with multiplication maps 
$M_{p,q} : X_p \otimes_A X_q \to X_{pq}$ such that 
\begin{enumerate}
\item $X_e$ is the standard bimodule ${}_A A_A$, and $M_{e,e} : A \otimes_A A \xrightarrow{\cong} A $ is simply multiplication on $A$;
\item if $p =e$, then $M_{e, q}: A \otimes_A X_q \xrightarrow{\cong}  \ol{A \cdot X_p}$ is the left action of $A$ on $X_q$;
\item if $q = e$ then $M_{p, e}:  X_p \otimes_A A \xrightarrow{\cong}   X_p$ is the right action of $A$ on $X_p$;

\item if $p, q \in P \setminus \{e\}$, then
$M_{p,q} : X_p \otimes_A X_q \xrightarrow{\cong} X_{pq}$;

\smallskip\item the multiplication maps are associative in the sense that
\[
M_{pq, r} (M_{p,q} \otimes \id_{X_r}) = M_{p, qr} (\id_{X_p} \otimes M_{q,r}) \foral p,q,r \in P.
\]
\end{enumerate}
Throughout this work we will also assume that  the left action of $A$ on $X_q$ is  nondegenerate (or essential) in the sense that $\ol{A \cdot X_q} =X_q$ for every $q\in P$,  and hence the multiplication map $M_{e,q}$ in \textup{(ii)} is an isomorphism of  $X_e\otimes_A X_q$ onto $X_q$.
\end{definition}
 \begin{remark}
We assume that the left action of $A$ is nondegenerate in order to be able to freely use the results from \cite{Seh18}.   Observe that, as pointed out in \cite[Remark~1.3]{KL19b}, this assumption is automatically satisfied when $P$ has a nontrivial unit $u$ because then
\[
 X_q =  X_u \otimes_A X_{u^{-1}q}  = X_u \otimes_A X_{u^{-1}} \otimes_A X_q= X_e\otimes X_q.
\]
It is plausible that one could extend the main results from \cite{Seh18} to product systems with degenerate left actions, and that this
would allow us to include such product systems in our results.
\end{remark}
 Henceforth we will be suppressing the use of the symbols $M_{p,q}$,
thus writing $\xi_p \xi_q$ for the image of $\xi_p \otimes \xi_q$ under $M_{p,q}$, and so
\[
\varphi_{pq}(a)(\xi_p \xi_q) = (\varphi_p(a) \xi_p) \xi_q \foral a \in A \text{ and } \xi_p \in X_p, \xi_q \in X_q.
\]
The product system structure gives rise to maps
\[
i_{p}^{pq} \colon \L X_{p} \to \L X_{pq}
\; \textup{ such that } \;
i_{p}^{pq}(S) (\xi_{p} \xi_{q})
=
(S \xi_{p}) \xi_{q}.
\]
If $x \in P^*$ then $i_{r}^{rx} \colon \L X_r \to \L X_{rx}$ is a $*$-isomorphism with inverse $i_{rx}^{rxx^{-1}} \colon \L X_{rx} \to \L X_{r}$.

%%%%%%%%%%%%%%%%%%%%%%%%%%%%%%%%
\begin{definition}
Let $P$ be a subsemigroup of a group $G$ and let $X$ be a product system over $P$. A \emph{Toeplitz representation} $t = \{t_p\}_{p\in P}$ of the product system $\{X_p \mid {p\in P}\}$ is a family of maps $t_p : X_p \rightarrow \B(H)$ such that $(t_e,t_p)$ is a Toeplitz representation of $X_p$ and
\[
t_p(\xi_p) t_q(\xi_q) = t_{pq}(\xi_p \xi_q) \foral \xi_p \in X_p, \xi_q \in X_q.
\]
The representation $t$ is said to be \emph{injective} if the homomorphism $t_e: X_e \to \B(H)$ is injective, in which case $t_p$ is isometric for each $p\in P$.
The \emph{Toeplitz algebra $\T(X)$ of $X$} is the universal C*-algebra generated by $X$ with respect to the Toeplitz representations of $X$.
The \emph{Toeplitz tensor algebra $\T(X)^+$ of $X$} is the norm-closed nonselfadjoint subalgebra of $\T(X)$ generated by $X$ and $A$.  Note that  in the present situation $X_e = A$, so $ \T(X)^+$ is generated by $X$ alone.
\end{definition}

A Toeplitz representation $t = \{t_p\}_{p\in P}$ induces a representation $t_{r,s}$ of $\K(X_s, X_r)$ on the same Hilbert space, determined by 
$t_{r,s}(\theta_{\xi_r, \xi_s}) = t_r(\xi_r) t_s(\xi_s)^*$.  In the case $s = r$ we slightly abuse the notation, as already indicated above for a single correspondence, and write $t_s$ in place of $t_{s,s}$.
This gives a representation triple $(t_r, t_{r,s}, t_s)$ of the bimodule $(\K X_r, \K(X_s, X_r), \K X_s)$.

\begin{proposition}\label{P:star inv LCM}
Let $P$ be a subsemigroup of a group $G$ and let $X$ product system over $P$.
Let $t = \{t_p\}_{p\in P}$ be a Toeplitz representation of $X$.
If $r \in P^*$ then 
\[
t_{r}(X_r)^* = t_{r^{-1}}(X_{r^{-1}}).
\]
If $w \in P$ and $r \in P^*$ then
\[
i_w^{wr}(k_w) \in \K X_{wr}
\text{ and }
t_{wr}(i_{w}^{wr}(k_w)) = t_w(k_w) \foral k_w \in \K X_w.
\]
\end{proposition}

\begin{proof}
Since $r \in P^*$ we have that $X_e \simeq X_r \otimes X_{r^{-1}}$.
Since $r^{-1} \in P^*$ we also have that $X_e \otimes X_{r^{-1}} \simeq X_{r^{-1}}$.
Hence we get the two equations 
\[
t_e(X_e) = \ol{t_r(X_r) t_{r^{-1}}(X_{r^{-1}})}
\qand 
\ol{t_e(X_e) t_{r^{-1}}(X_{r^{-1}})} = t_{r^{-1}}(X_{r^{-1}}).
\]
By multiplying the first equation on the left with $t_{r}(X)^*$, and by using the second equation, we get
\[
 t_{r}(X_{r})^* t_e(X_e)  \subseteq \ol{t_r(X_r)^* t_r(X_r) t_{r^{-1}}(X_{r^{-1}})} \subseteq \ol{t_e(X_e) t_{r^{-1}}(X_{r^{-1}})} = t_{r^{-1}}(X_{r^{-1}}).
\]
Replacing $r^{-1}$ by $r$ in the second equation and taking adjoints gives $t_{r}(X_{r})^* = \ol{t_{r}(X_{r})^* t_e(X_e)}$, hence
\[t_{r}(X_{r})^*   \subset t_{r^{-1}}(X_{r^{-1}}).\] 
%Taking adjoints gives $t_r(X_r) \subseteq t_{r^{-1}}(X_{r^{-1}})^*$.
As this holds for arbitrary $r \in P^*$ it also holds for $r^{-1} \in P^*$ and so 
\[
t_{r^{-1}}(X_{r^{-1}})^* \subseteq t_{r}(X_r).
\]
By applying adjoints we get the required reverse inclusion.

Since $r \in P^*$ we have that $\iota_w^{wr} \colon \L X_w \to \L X_{wr}$ is a $*$-isomorphism and thus it preserves the compact operators.
Therefore $\iota_w^{wr}(k_w) \in \K X_{wr}$ for $k_w \in \K X_w$.
By applying on elementary tensors we see that $t_{wr}(i_w^{wr}(k_w))$ coincides with $t_w(k_w)$ restricted on $\ol{t_{wr}(\K X_{wr}) H}$, where $H$ is the Hilbert space where $t = \{t_p\}_{p\in P}$ acts on.
On the other hand $t_w(k_w)$ is completely defined by its representation on $\ol{t_w(\K X_w) H}$.
However by the first part we have that
\[
t_{wr}(\K _{wr}) = \ol{t_w(X_w) t_r(X_r) t_r(X_r)^* t_w(X_w)^*} = \ol{t_w(X_w) t_e(X_e) t_w(X_w)^*} = t_w(\K X_w),
\]
and so $t_{wr}(i_{w}^{wr}(k_w)) = t_w(k_w)$.
\end{proof}

The Fock space representation $\overline{t}$ of Fowler \cite{Fow02} ensures that $X$, embeds isometrically in $\T(X)$. It is given as follows. Let $\F(X) = \oplus_{r \in P} X_r$ and for $\xi_p \in X_p$ define
\[
\ol{t}_p(\xi_p) \eta_r = \xi_p \eta_r
\qforal
\eta_r \in X_r.
\]
Then $\ol{t}:=\{\ol{t}_p\}$ defines a representation of $X_p$ for every $p \in P$, and induces a representation of $\T(X)$. By taking the compression of $\ol{t}_e$ at the $(e, e)$-entry we see that $\ol{t}_e$ is injective, and hence $\ol{t}_p$ is injective for each $p\in P$.

%%%%%%%%%%%%%%%%%%%%%%%%%%%%%%%%
\begin{definition}
Let $P$ be a  subsemigroup of a group $G$ and let $X$ a product system over $P$.
The \emph{Fock algebra} $\T_\la(X)$ is the C*-algebra generated by the Fock representation.
The \emph{Fock tensor algebra} $\T_\la(X)^+$ is the subalgebra of $\T_\la(X)$ generated by $X$.
\end{definition}

\subsection{Product systems over right LCM semigroups} 
A semigroup $P$ is said to be a \emph{right LCM semigroup} if it is left cancellative and satisfies \emph{Clifford's condition} \cite{Law12, Nor14}: 
\begin{center}
for every $p, q \in P$ with $p P \cap q P \neq \mt$ there exists a $w \in P$ such that $p P \cap q P = w P$.
\end{center}
In other words, if $p, q \in P$ have a right common multiple then they have a right Least Common Multiple.
We always assume that the semigroup $P$ is contained in a group, and that it contains the identity element. It follows that $P$ is by default cancellative, and we will refer to $P$ simply as a \emph{right LCM subsemigroup of a group}.
It is clear that {if an element  $w\in P$} is a right Least Common Multiple for $p, q \in P$  then so is $w x$  for every unit 
$x \in P^*:= P\cap P^{-1}$.
Right LCM semigroups that embed in a group and have no nontrivial units, so that least common multiples are unique whenever they exist,
have been called  \emph{weak quasi-lattice ordered} semigroups in \cite{ABCD2019}.

\begin{example}
Right LCM subsemigroups of groups include as primary examples the quasi-lattice orders defined in \cite{Nic92}.  Several noteworthy examples beyond quasi-lattice orders  have been considered recently in the context of isometric representations.  We would like to list a few here in order to illustrate the variety. Since we do not require any specialized knowledge of these examples in this paper, we limit ourselves to giving references where details can be found.

The inclusion of an Artin monoid in its corresponding Artin group is always a right LCM subsemigroup of a group 
\cite{Brieskorn-Saito}.  Artin monoids have trivial unit groups so they  actually determine weak quasi-lattice orders in their respective groups. At this point, only the  particular cases of Artin monoids of spherical and rectangular type are actually known to be 
true quasi-lattice orders, \cite{Brieskorn-Saito} and \cite{Crisp-Laca2002}.

Another important class of right LCM subsemigroups of groups are the
 inclusions of  Baumslag-Solitar monoids $B(m,n)^+$ in their corresponding Baumslag-Solitar groups  
 $B(m,n):= \langle a,b \mid ab^m = b^n a\rangle$. These monoids always have trivial unit groups, and they give quasi-lattice orders in their groups if and only if either $mn>0$  \cite[Theorem 2.11]{Spi12}
 or else $mn < 0$  and $|m|=1$ \cite[Lemma 2.12]{Spi12}, see also 
 \cite[Example~3.5]{ABCD2019}. The remaining case is particularly interesting because 
when $mn<0$ and $|m| \neq 1$, no group embedding of $B(m,n)^+$ can be a quasi-lattice order. This is proved directly in  \cite[Proposition 3.10]{ABCD2019} under the extra assumption that $m$ does not divide $n$; a proof without this assumption can be derived from the failure of the Toeplitz condition in this case, as shown in  \cite[Section~4.2]{Li2020}.

There are also right LCM semigroups such as $\mathbb Z\rtimes \mathbb N^{\times}$, 
in which the whole additive part $\mathbb Z \times \{1\}$
consists of units. A sizable general class of examples like this arises from considering the semigroup $R\rtimes R^\times$ of affine 
transformations of an integral domain $R$ \cite{Li13}. The hypothesis of integral domain is necessary for 
these semigroups to embed in groups; which can be taken to be  groups of  affine 
transformations of the corresponding  fields of fractions. The semigroup $R\rtimes R^{\times}$ is a right LCM 
semigroup if and only if $R$ satisfies the GCD condition \cite[Proposition 2.23]{Nor14}. The best known  examples are the $ax+b$ semigroups  
of the rings of algebraic integers in number fields of class number $1$. In these, the groups of units are nonabelian and consist
of the semidirect products of the additive group of those rings by the multiplicative action of the units. 

Partly inspired by these, more right LCM semigroups that embed in a group have been constructed 
as semidirect products associated to certain algebraic actions of $\mathbb N^k$ on an abelian group
that respect the order in the sense of \cite[Definition 8.1]{BLS18}. 
\end{example}

%%%%%%%%%%%%%%%%%%%%%%%%%%%%%%%%

It will be convenient for us to work with finite subsets $F$ of $P$ 
on which a `local' right LCM operation can be defined that reflects the structure of upper bounds in $P$, thus generalizing the  notion of 
$\vee$-closed sets used in the case of a quasi-lattice order.  
The problem is that  expressions like $p\vee q$ or $ \lcm(p,q)$  customarily used to denote smallest common upper bound or least common multiple of two elements, are not well defined for a general right LCM  semigroup 
because of the nonuniqueness caused by the presence of nontrivial units in $P$. So we need to impose restrictions on $F$ to ensure uniqueness.
\begin{definition}\label{def:veeclosed}
Let $P$ be a right LCM subsemigroup of a group $G$. 
A finite subset $F$ of $P$ is said to be \emph{{$\vee$-closed}} if for every $p,q \in F$ with $p P \cap q P \neq \mt$ there exists a unique $w \in F$ such that $p P \cap q P = w P$, equivalently, $F$ contains exactly one right LCM for any two of its elements that have a right LCM in $P$.
\end{definition}
An easy example that illustrates this notion is provided by the multiplicative semigroup $P= \mathbb Z\setminus \{0\}$, in which least common multiples are defined only up to $\pm 1$. But, for instance, the set $F$ of all positive divisors of a positive integer is $\vee$-closed because any two  elements of $F$ have a unique l.c.m. in $F$.

 Another way to approach this  is to realize that a finite subset $F\subseteq P$ is 
 \emph{{$\vee$-closed}}   iff  the restriction of the `right ideal map' $\I:p\mapsto pP$
 to $F$ gives a bijection whose image $\I(F) := \{p P \mid p \in F\}$  is closed under intersection.
It is then easy to see that if $F$ is $\vee$-closed, then the familiar relation
\[
p \leq q \Leftrightarrow q^{-1}p \in  P
\]
actually defines a partial order on $F$, and hence, being finite,  each {$\vee$-closed} set has maximal and minimal elements.

Following Fowler's work \cite{Fow02}, Brownlowe, Larsen and Stammeier \cite{BLS18}, and Kwasniewski and Larsen \cite{KL19a, KL19b} considered product systems over right LCM semigroups.

%%%%%%%%%%%%%%%%%%%%%%%%%%%%%%%%
\begin{definition}
A product system $X$ over a right LCM semigroup $P$ with coefficients from $A$ is called \emph{compactly aligned} if for $p, q \in P$ with $p P \cap q P = w P$ we have that
\[
i_{p}^{w}(k_p) i_{q}^{w}(k_q) \in \K X_{w} \text{ whenever } k_p \in \K X_{p}, k_q \in \K X_{q}.
\]
\end{definition}
A note is in order for clarifying that this is independent of the choice of $w$.
Recall that if $w'$ is a right LCM of $p, q$ then $w' = wx$ for some $x \in P^*$.
Since $\L X_{w} \simeq \L X_{wx}$ we have that $i_{p}^{w}(k_p) i_{q}^{w}(k_q) \in \K X_w$ if and only if $i_{p}^{wx}(k_p) i_{q}^{wx}(k_q) = i^{wx}_w (i_{p}^{w}(k_p) i_{q}^{w}(k_q)) \in \K X_{wx}$ for all $x \in P^*$.

%%%%%%%%%%%%%%%%%%%%%%%%%%%%%%%%
\begin{definition}
Let $P$ be a right LCM subsemigroup of a group $G$ and let $X$ be a compactly aligned product system over $P $ with coefficients in $A$.
A \emph{Nica-covariant representation $t = \{t_p\}_{p\in P}$} is a Toeplitz representation of $X$ that in addition satisfies the \emph{Nica-covariance  condition}: for all $k_p \in \K X_{p}$ and $k_q \in \K X_{q}$, i.e., 
\[
t_{p}(k_p) t_{q}(k_q) = 
\begin{cases}
t_{w} (i_{p}^{w}(k_p) i_{q}^{w}(k_q)) & \text{ if } p P \cap q P = w P, \\
0 & \text{ otherwise}.
\end{cases}
\]
The \emph{Nica-Toeplitz algebra $\N\T(X)$ of $X$} is the universal C*-algebra generated by $X$ with respect to the Nica-covariant representations of $X$.
The \emph{Nica-Toeplitz tensor algebra $\N\T(X)^+$ of $X$} is the norm closed nonselfadoint subalgebra of $\N\T(X)$ generated by $X$.
\end{definition}

%%%%%%%%%%%%%%%%%%%%%%%%%%%%%%%%
Notice that in the definition of Nica-covariance, the choice of the least common multiple is arbitrary. This is because Proposition \ref{P:star inv LCM} implies that 
 \[
t_{w} (i_{p}^{w}(k_p) i_{q}^{w}(k_q))
=
t_{wx} (i_{p}^{wx}(k_p) i_{q}^{wx}(k_q)), \,\,
k_p \in \K X_p, k_q \in \K X_q,
\]
provided that $pP \cap qP = wP$ and $x \in P^*$.

%%%%%%%%%%%%%%%%%%%%%%%%%%%%%%%%
Under the assumption of compact alignment, one can check that the Fock representation is automatically Nica-covariant.
Thus $\N\T(X)$ is non-trivial. In the case where $P = \bZ_+$ we actually have that $\N\T(X) = \T(X)$. This is not necessarily true for other right LCM semigroups. Indeed in the case where $P = Z^n_+$, $n \geq 2$, 
Dor-On and Katsoulis provide a counterexample to this effect in \cite[Example 5.2]{DK20}.
The same example further shows that $\T(X)^+$  need not be completely isometric to $\N\T(X)^+$.

Our next goal is to understand the cores of Nica-covariant representations of $X$. So let $t=\{t_p\}_{p \in P}$ be a Nica-covariant representation of $X$.
We compute
\[
t_p(X_p)^* t_p(X_p) \cdot t_p(\xi_p)^* t_q(\xi_q) \cdot t_q(X_q)^* t_q(X_q)
\subseteq
\ol{t_p(X_p)^* t_p(\K X_p) t_q(\K X_q) t_q(X_q)}
\]
and then take a limit by an approximate identity in $\ol{t_p(X_p)^* t_p(X_p)}$ and in $\ol{t_q(X_q)^* t_q(X_q)}$, to derive that
\[
t_p(\xi_p)^* t_q(\xi_q) \in \ol{t_{p'}(X_{p'}) t_{q'}(X_{q'})^*} \qfor w P = pP \cap q P, p' = p^{-1} w, q' = q^{-1} w,
\]
and
\[
t_p(\xi_p)^* t_q(\xi_q) = 0 \qfor pP \cap qP = \mt.
\]
Hence the C*-algebra $\ca(t)$ generated by $\{t_p(X_p)\}_{p\in P}$ is given by
\[
\ca(t) = \ol{\spn}\{t_p(\xi_p) t_q(\xi_q)^* \mid \xi_p \in X_p, \xi_q \in X_q \textup{ and } p,q \in P\}.
\]
If $F\subseteq P$, then we write
\begin{equation}\label{eq:BFtdef}
B_{F, t} 
:= \ol{\spn} \{ t_{p}(k_p) \mid k_p \in \K X_p, p \in F \}
\end{equation}

By definition, the \emph{core} of the representation $t$ is  the set $B_{P, t}$, which is clearly given by 
 \[
 B_{P, t} = \ol{\bigcup \{ B_{F, t} \mid F \subseteq P \text{ finite}\}}.
 \] 
 Notice that when $\I(F)$ is closed under intersections, Nica-covariance implies that $B_{F,t}$ is a $\ca$-subalgebra of $\ca(t)$.  
 We wish to show next that if we restrict the above union to $\vee$-closed sets $F$ we still obtain $B_{P, t}$ and that 
  for $\vee$-closed sets $F$  the linear spans in \eqref{eq:BFtdef} are automatically closed. 
 We begin with this last claim.

%%%%%%%%%%%%%%%%%%%%%%%%%%%%%%%%
\begin{proposition}
Let $P$ be a right LCM subsemigroup of a group $G$, let  $X$ be a compactly aligned product system over $P$ and let $t=\{t_p\}_{p \in P}$ be a Nica-covariant representation of $X$. 
If $F \subseteq P$ is a {$\vee$-closed} set, then 
\[
B_{F, t} = \spn\{t_p(k_{p}) \mid k_p \in \K X_p, p \in F\}.
\]
\end{proposition}

\begin{proof}
It suffices to prove the above in the case where $t = \wh{t}$ the universal representation $\wh{t}$ of $\N\T(X)$.
Let $\ol{t}$ be the Fock representation of $X$ and $\Phi \colon \N\T(X) \to \ca(\ol{t})$ be the canonical $*$-epimorphism.
Let $f = \lim_i f_i$ for a net $(f_i)$ with
\[
f_i = \sum_{p \in F} \wh{t}_p(k_{p,i}) \in  \spn\{\wh{t}_p(k_{p}) \mid k_p \in \K X_p, p \in F\}.    
\]
Then we also have that
\[
\Phi(f) = \lim_i \sum_{p \in F} \ol{t}_p(k_{p,i}).
\]
Recall that $F$ is $\vee$-closed, so that $P$ induces a partial order on $F$, and let $p_0 \in F$ be a minimal element of $F$ in this partial order.  Take the compression to the $(p_0, p_0)$-entry by the projection $Q_{p_0} \colon \F(X) \to X_{p_0}$.
Then we have that
\[
\lim_i k_{p_0,i} = \lim_i Q_{p_0} f_i Q_{p_0} = Q_{p_0} \Phi(f) Q_{p_0}.
\]
Therefore the net $(k_{p_0,i} )$ is convergent in $\K X_{p_0}$, say to some $k_{p_0}$, and so $\lim_i \wh{t}_{p_0}(k_{p_0,i} ) = \wh{t}_{p_0}(k_{p_0})$.
We repeat for $f - \wh{t}_{p_0}(k_{p_0})$ and the net $(f_i - \wh{t}_{p_0}(k_{p_0,i} ))$, and for the {$\vee$-closed} set $F' = F \setminus \{p_0\}$.
Proceeding inductively, we see that for every $p \in F$ there exists a $k_p$ such that $\lim_i k_{p, i}= k_p$ and this shows that $f \in 
\spn\{\wh{t}_p(k_{p}) \mid k_p \in \K X_p, p \in F\}$, which completes the proof.
\end{proof}

 Next we see that $\vee$-closed sets suffice to  generate the core.

%%%%%%%%%%%%%%%%%%%%%%%%%%%%%%%%
\begin{proposition}\label{P:purified}
Let $P$ be a right LCM subsemigroup of a group $G$. Let $X$ be a compactly aligned product system over $P$ and let $t=\{t_p\}_{p \in P}$ be a Nica-covariant representation of $X$. 
Then
\[
B_{P, t} 
=
\ol{\bigcup \{ B_{F, t} \mid F \subseteq P \text{ finite and {$\vee$-closed}} \}}.
\]
\end{proposition}

\begin{proof}
Suppose that $F$ is an arbitrary finite subset of $P$.  We first saturate $\I(F):= \{pP: p\in F\}$ under intersections,
\[
\I(F)^\cap:= \big\{ \bigcap_{p\in F'}pP :  \emptyset \neq F' \subset F\big\},
\]
 and then choose a unique generator of each principal ideal in the resulting set $\I(F)^\cap$. This gives 
a \emph{$\vee$-closed} set $F^\vee$ with the property that $ \I(F) \subseteq \I(F)^\cap =\I(F^\vee)$. 
We note in passing that there is no uniqueness in this process because  the  choice of generators is arbitrary.
The result will follow once we show that
$$B_{F, t} \subseteq B_{F^{\vee}, t}.$$
This is not obvious because $F$ itself may not be contained in $F^\vee$, so we need to verify that we do not lose any part of $B_{F, t}$ 
 when we restrict to unique generators for the ideals in $\I(F)^\cap$. 
Towards this end, since $\I(F)^\cap =\I(F^\vee)$, it  is enough to show that given $p, q \in P$ with $p P = q P$, then $t_{q}(\K X_q) = t_p( \K X_p)$.
In that case there exists  a unit $r \in P^*$ such that $q = pr$.
By Proposition \ref{P:star inv LCM} we have that $t_r(X_r)^* = t_{r^{-1}}(X_{r^{-1}})$ and $\ol{t_r(X_r) t_{r^{-1}}(X_{r^{-1}})} = t_e(X_e)$.
Hence
\[
t_{q}(\K X_q) = \ol{t_p(X_p) t_r(X_r) t_{r}(X_r)^* t_p(X_p)^*} = t_p(\K X_p).
\]
This completes  the  proof.
\end{proof}

%%%%%%%%%%%%%%%%%%%%%%%%%%%%%%%%
\section{Cosystems and their C*-envelopes} \label{S;cosystem}
%%%%%%%%%%%%%%%%%%%%%%%%%%%%%%%%

In what follows $G$ will always denote a fixed discrete group.
We write $u_g$ for the generators of the universal group C*-algebras $\ca(G)$ and $\la_g$ for the generators of the left regular representation $\ca_\la(G)$.
We write $\la \colon \ca(G) \to \ca_\la(G)$ for the canonical $*$-epimorphism.
Recall that $\ca(G)$ admits a faithful $*$-homomorphism
\[
\De \colon \ca(G) \to \ca(G) \otimes \ca(G) 
\text{ such that }
\De(u_g) = u_g \otimes u_g
\]
given by the universal property of $\ca(G)$; the left inverse of $\Delta$ is  given by $\id \otimes \chi$, where $\chi: \ca(G) \to \bC$ is the representation arising from the trivial character of $G$, and we identify $ \ca(G) \otimes \bC $ with $\ca(G) $.

%%%%%%%%%%%%%%%%%%%%%%%%%%%%%%%%
\begin{definition}\label{D:cis coa}
Let $\A$ be an operator algebra.
A \emph{coaction of $G$ on $\A$} is a completely isometric representation $\de \colon \A \to \A \otimes \ca(G)$ such that $\sum_{g \in G} \A_g$ is norm-dense in $\A$ for the  spectral subspaces
\[
\A_g := \{a \in \A \mid \de(a) = a \otimes u_g\}. 
\]
If, in addition, the map $(\id \otimes \la) \de$ is injective then the coaction $\de$ is called \emph{normal}. 

 A map $\de$ as in Definition \ref{D:cis coa} automatically satisfies the coaction identity
\begin{equation} \label{eq;coactionid}
(\de \otimes \id_{\ca(G)}) \de = (\id_{\A} \otimes \De) \de.
\end{equation}
Indeed, (\ref{eq;coactionid}) is readily seen to hold on $\A_g$, $g \in G$ and therefore on $\A$, since $\sum_{g \in G} \A_g$ is norm-dense in $\A$.

If $\A$ is an operator algebra and $\de \colon \A \to \A \otimes \ca(G)$ is a  coaction on $\A$, then we will refer to the triple $(\A, G, \de)$ as a  \emph{cosystem}. 
A map $\phi \colon \A \to \A'$ between two cosystems $(\A, G, \de)$ and $(\A', G, \de')$ is said to be \emph{$G$-equivariant}, or simply \emph{equivariant}, if $\de' \phi=(\phi\otimes \id)\de$.
\end{definition}

It follows from the definition that if $(\A, G, \de)$ is a cosystem then
\[
\A_g \cdot \A_h \subseteq \A_{g h} \foral g, h \in G,
\]
because $\de$ is a homomorphism.
Conversely, if there are subspaces $\{\A_g\}_{g \in G}$ such that $\sum_{g \in G} \A_g$ is norm-dense in $\A$ and a representation $\de \colon \A \to \A \otimes \ca(G)$ such that
\[
\de(a_g) = a_g \otimes u_g \foral a_g \in \A_g, g \in G,
\]
then $\de$ is a coaction of $G$ on $\A$.
Indeed $\de$ satisfies the coaction identity and it is completely isometric since $(\id_{\A} \otimes \chi) \de = \id_{\A}$.

%%%%%%%%%%%%%%%%%%%%%%%%%%%%%%%%
\begin{remark}\label{R:nd cis}
Let $(\A, G, \de)$ be a cosystem and assume that $\de$ extends to a $*$-homomorphism $\de \colon \ca(\A) \to \ca(\A) \otimes \ca(G)$ that satisfies the coaction identity
\[
(\de \otimes \id) \de(c) = (\id \otimes \De) \de(c) \foral c \in \ca(\A).
\]
Then $\de$ is automatically non-degenerate on $\ca(\A)$ in the sense that
\[
\ol{\de(\ca(\A)) \left[\ca(\A) \otimes \ca(G)\right]} = \ca(\A) \otimes \ca(G).
\]
Indeed if $(e_i)$ is a contractive approximate identity  for $\ca(\A)$ then we can write
\[
a_{g_1} a_{g_2}^*a_{g_3}  \cdots a_{g_{n-1}} a_{g_n}^* \otimes u_h
=
\lim_i (a_{g_1} \otimes u_{g_1}) \cdots (a_{g_n} \otimes u_{g_n})^* \left(e_i \otimes u_{(g_1g_2^{-1} \cdots g_{n-1}^{-1} g_n)^{-1} h}\right) ,
\]
and likewise for all products of the form
\[
a_{g_2}^*a_{g_3} \cdots a_{g_{n-1}} a_{g_n}^*, \ 
a_{g_2}^* a_{g_3}\cdots a_{g_n}^* a_{g_{n+1}} \ 
\text{ and } \ 
a_{g_1} a_{g_2}^* \cdots a_{g_n}^* a_{g_{n+1}}
\]
 in $\ca(\A)$.
By definition of $\de$ these products generate $\ca(\A)$.
\end{remark}

%%%%%%%%%%%%%%%%%%%%%%%%%%%%%%%%
\begin{remark} \label{r;Qui}
If the operator algebra $\A$ happens to be a $\ca$-algebra, Definition \ref{D:cis coa} coincides with the definition given by Quigg in \cite[Section 1]{Qui96}.
In that case $\de$ is a faithful $*$-homomorphism and we have that
\[
(\A_g)^* = \{a^* \in \A \mid \de(a^*) = a^* \otimes u_{g^{-1}} \} = \A_{g^{-1}}.
\]
%An argument a
 As in Remark \ref{R:nd cis} the coaction is then non-degenerate, i.e., it is a \emph{full} coaction. Furthermore, there exists a conditional expectation $E_{\delta}:  \A \rightarrow \A_e$ vanishing on $\A_g$, for all $g \in G\backslash \{  e \}$; see \cite[Proposition A.4]{Seh18} for a proof. Therefore if $\A_0\subseteq \A$ is a dense subset of $\A$, then $E_{\de}(\A_0)$ is a dense subset of $\A_e$.
\end{remark}

For our next result, we use Fell's absorption principle to give sufficient conditions for the existence a compatible  normal  coaction.
%%%%%%%%%%%%%%%%%%%%%%%%%%%%%%%%
%%%%%%%%%%%%%%%%%%%%%%%%%%%%%%%%
\begin{proposition}\label{P:Fell ind}
Let $\A$ be an operator algebra and let $G$ be a group.
Suppose there are subspaces $\{\A_g\}_{g \in G}$ such that $\sum_{g \in G} \A_g$ is norm-dense in $\A$, and there is a completely isometric homomorphism
\[
\de_\la \colon \A \longrightarrow \A \otimes \ca_\la(G)
\]
such that
\begin{equation}\label{E:reduced}
\de_\la(a) = a \otimes \la_g \foral a \in \A_g, g \in G.
\end{equation}
Then $\A$ admits a normal coaction $\de$ of $G$ satisfying  $\de_\la = (\id \otimes \la) \de$.
\end{proposition}
\begin{proof} That $\de_\la$  satisfies the coaction identity follows easily from \eqref{E:reduced} and  the subspaces
\[
\{a\in \A \mid \de_\la(a) = a\otimes \la_g\}, \qquad g\in G
\]%of $\de_\la$ 
have dense linear span because they contain the corresponding $\A_g$ for every $g\in G$.
 We need to show that there is a  completely isometric homomorphism
\[
\de \colon \A \longrightarrow \A \otimes \ca(G)
\]
such that $\de_\la = (\id \otimes \la) \de$.
Let 
\[
\phi \colon \ca_\la(G) \longrightarrow \ca_\la(G) \otimes I \longrightarrow \ca_\la(G) \otimes \ca(G) : \la_g \mapsto \la_g \otimes u_g
\]
be the $*$-isomorphism given by Fell's absorption principle.
We can then define 
\[
\de := (\de_\la^{-1} \otimes \id_{\ca(G)}) (\id_{\A} \otimes \phi) \de_\la,
\]
which is the desired completely isometric homomorphism.
Indeed, 
\[
\de(a_g) = (\de_\la^{-1} \otimes \id_{\ca(G)})((a_g \otimes \la_g) \otimes u_g) = a_g \otimes u_g
\]
for every $a_g \in \A_g$, and thus $\de$ satisfies the coaction identity and the spectral subspaces of $\de$ have dense linear span because they contain the subspaces $\A_g$.
Since $\de_\la = (\id_{\A} \otimes \la) \de$ is faithful we have that $\de$ is a normal coaction on $\A$, and the proof is complete.
\end{proof}

%%%%%%%%%%%%%%%%%%%%%%%%%%%%%%%%
%%%%%%%%%%%%%%%%%%%%%%%%%%%%%%%%
\begin{example} \label{E;normalco}
The reduced group C*-algebra $\ca_\la(G)$ admits a faithful $*$-homomophism
\[
\De_\la \colon \ca_\la(G) \longrightarrow \ca_\la(G) \otimes \ca_\la(G) 
\text{ such that }
\De_\la(\la_g) = \la_g \otimes \la_g.
\]
Indeed consider the unitary
\[
U \colon \ell^2(G) \otimes \ell^2(G) \longrightarrow \ell^2(G) \otimes \ell^2(G)
\text{ with }
U (e_r \otimes e_s) = e_r \otimes e_{r s},
\]
and verify that 
\[
( \la_g \otimes \la_g ) U = U (\la_g \otimes I)
\foral 
g \in G.
\]
Therefore $\ad_U$ implements a faithful $*$-homomorphism
\[
\ca_\la(G) \simeq \ca(\la_g \otimes I \mid g \in G) \stackrel{\ad_U}{\longrightarrow} \ca(\la_g \otimes \la_g \mid g \in G).
\]
Thus $\ca_\la(G)$ admits a normal coaction of $G$.
\end{example}

%%%%%%%%%%%%%%%%%%%%%%%%%%%%%%%%
\begin{definition}\label{D:coaction}
Let $(\A, G, \de)$ be a cosystem.
A triple $(C', \iota', \de')$ is called a \emph{C*-cover} for $(\A, G, \de)$ if  $(C', G, \de')$ forms a cosystem and $(C', \iota')$ forms a C*-cover of $\A$ with $\iota : \A\rightarrow C'$ being equivariant.
\end{definition}

%%%%%%%%%%%%%%%%%%%%%%%%%%%%%%%%
\begin{definition}
Let $(\A, G, \de)$ be a cosystem.
The \emph{C*-envelope of $(\A, G, \de)$} is a C*-cover for $(\A, G, \de)$, denoted  by $( \cenv(\A, G, \de), \iotenv, \delenv)$,
that satisfies the following property: 
for any other C*-cover $(C', \iota', \de')$ of  $(\A, G, \de)$ there exists an equivariant  $*$-epimorphism $\phi \colon C' \to \cenv(\A, G, \de)$ that makes the following diagram
\[
\xymatrix{
& & C' \ar@{.>}[d]^{\phi} \\
\A \ar[rru]^{\iota'} \ar[rr]^{\iotenv} & & \cenv(\A, G, \de)
}
\]
commutative. We will often omit the embedding $\iotenv$ and the coaction $\delenv$  and  refer to the triple simply  as $ \cenv(\A, G, \de)$.
\end{definition}

As in the case of the C*-envelope for an operator algebra, it is easily seen that if the C*-envelope for a cosystem exists, then it is unique up to a natural notion of  isomorphism for cosystems.
The following theorem verifies  that  every  cosystem 
has a C*-envelope and gives a concrete picture for it.

%%%%%%%%%%%%%%%%%%%%%%%%%%%%%%%%
\begin{theorem}\label{T:co-env}
Let $(\A, G, \de)$ be a cosystem and let $\iota \colon \A \to \cenv(\A)$ be the natural inclusion map.
Then the triple
\[
\left( \ca(\iota(a_g) \otimes u_g \mid g \in G), \ (\iota \otimes \id_{\ca(G)} ) \de,\  \id \otimes \De \right)
\]
is (isomorphic to) the C*-envelope for the cosystem $(\A, G, \de)$.
\end{theorem}
\begin{proof}
Let $(\A, G, \de)$ be a cosystem and fix the embedding $\iota \colon \A \to \cenv(\A)$. 
By considering the  composition
\begin{equation} \label{eq;correction}
\xymatrix@C=2cm{
\A \ar[r]^{\de \phantom{ooo} } & \A \otimes \ca(G) \ar[r]^{\iota \otimes \id_{\ca(G)} \phantom{oooo} } \ar[r] & \cenv(\iota(\A)) \otimes \ca(G),
}
\end{equation}
and recalling that the minimal tensor product of completely isometric representations is completely isometric, 
we see that the C*-algebra $$\ca(\iota(a_g) \otimes u_g \mid g \in G)$$ is a C*-cover for $\A$.
We can then endow it with the coaction $\id \otimes \De$; ) this makes $(\iota \otimes \id_{\ca(G)}) \de$ into a $\de$--$(\id \otimes \De)$ equivariant  homomorphism and the triple
\[
(\ca(\iota(a_g) \otimes u_g \mid g \in G), (\iota \otimes \id_{\ca(G)}) \de, \id \otimes \De)
\]
becomes a C*-cover for  $(\A, G, \de)$.
Next let $(C', \iota', \de')$ be a C*-cover and let $\phi \colon C' \to \cenv(\A)$ be as in (\ref{eq;env}).
We see that the following diagram is commutative
\[
\xymatrix@R=1cm@C=1cm{
%line 1
C' \ar[rr]^{\de' \phantom{o} } \ar@{.>}[drr] & & 
C' \otimes \ca(G) \ar[rr]^{\id \otimes \De} \ar[d]^{\phi \otimes \id} \ar@{.>}[drr] & & 
C' \otimes \ca(G) \otimes \ca(G) \ar[d]^{\phi \otimes \id \otimes \id} \\
%line 2
& & 
\ca(\iota(\A)) \otimes \ca(G) \ar[rr]^{\id \otimes \De \phantom{ooo} } & & 
\ca(\iota(\A)) \otimes \ca(G) \otimes \ca(G)
}
\]
as it is a diagram of $*$-epimorphisms that agree on the copies of $\A$.
We then obtain the canonical equivariant $*$-epimorphism
\[
(\phi \otimes \id_{\ca(G)}) \de' \colon (C', \iota', \de') \longrightarrow (\ca(\iota(a_g) \otimes u_g \mid g \in G), (\iota \otimes \id_{\ca(G)}) \de, \id \otimes \De)
\]
by following the diagonal arrows and using the coaction identity on $\de'$.
Indeed a direct computation shows that
\[
(\id \otimes \De) \left( \left( \phi \otimes \id \right) \de' \right)
=
(\phi \otimes \id \otimes \id) (\id \otimes \De) \de'
=
\left( \left( \left(\phi \otimes \id \right) \de' \right) \otimes \id \right) \de',
\]
and the proof is complete.
\end{proof}

Under certain hypothesis, one can obtain a more concrete picture for the C*-envelope. 
%%%%%%%%%%%%%%%%%%%%%%%%%%%%%%%%%%%%%

\begin{corollary} \label{C:normal} 
Let $(\A, G ,\de)$ be a normal cosystem, and let $\de_{\la}: \A \rightarrow \A\otimes\ca_{\la}(G)$ be a completely isometric homomorphism  satisfying the assumptions of Proposition~\ref{P:Fell ind} with respect to the spectral subspaces of the coaction $\de$.   If $\ol{\Delta_r}:   \ca_{\la}(G)\rightarrow  \ca_{\la}(G)  \otimes  \ca(G)$ denotes the normal coaction  implied by Example~\ref{E;normalco}, then
\begin{equation} \label{eq;normalenv}
\Big(\cenv(\A, G, \de) , \iotenv, \delenv \Big) \simeq  \left( \ca(\iota(a_g) \otimes \la_g \mid g \in G),\  (\iota \otimes \id_{\ca_{\la}(G)} ) \de_{\la},\ \id \otimes \ol{\Delta_r}\right).
\end{equation}
In particular, the coaction $\delenv$ on $\cenv(\A, G, \de)$ is normal.
\begin{proof}
By Theorem~\ref{T:co-env} the C*-envelope of $(\A, G, \de)$ is given by 
\[ 
\left( \ca(\iota(a_g) \otimes u_g \mid g \in G),\  (\iota \otimes \id_{\ca(G)} ) \de, \ \id \otimes \De \right).
\]
Since the right side of (\ref{eq;normalenv}) is a $\ca$-cover for $(\A, G ,\de)$, the defining properties of the C*-envelope imply an equivariant $*$-homomorphism
\[
\phi:  \ca(\iota(a_g) \otimes \la_g \mid g \in G) \longrightarrow  \ca(\iota(a_g) \otimes u_g \mid g \in G).
\]
If $q: \ca(G)\rightarrow \ca_{\la}(G)$ is the natural quotient then $(\id\otimes q) |_{\ca(\iota(a_g) \otimes u_g \mid g \in G)}$ provides an inverse for $\phi$ and the  conclusion follows.  
\end{proof}
\end{corollary}

%%%%%%%%%%%%%%%%%%%%%%%%%%%%%%%%
 By duality, a coaction of a discrete abelian group $G$  on an operator algebra $\A$ corresponds to a point-norm continuous action $\{\be_\ga\}_{\ga \in \wh{G}}$ of the dual group $\wh{G}$ on $\A$.
Since each $\be_\ga$ is a completely isometric automorphism, it extends to an automorphism $\tilde{\be}_\ga$ of the C*-envelope $\cenv(\A)$, 
yielding a point norm continuous action $\{\tilde{\be}_\ga \}_{\ga \in \wh{G}}$ of  $\wh{G}$ on $\cenv(\A)$. Again by duality, this corresponds to a coaction of $G$ on $\cenv(\A)$.
Hence, for abelian $G$, the C*-envelope for a cosystem coincides with the usual C*-envelope, equipped with the coaction given by  (the duals of)  the group of extended automorphisms $\{ \tilde\be_\ga\mid \ga\in \hat G\}$.
Equivalently  every coaction of a discrete abelian group on an operator algebra lifts to a coaction  on its C*-envelope.
It is not known if this is the case for more general classes of groups. 

%%%%%%%%%%%%%%%%%%%%%%%%%%%%%%%%
\begin{corollary}\label{C:fpa 1-1}
Let $(\cenv(\A, G, \de), \iota, \de_{\env})$ be the C*-envelope for a cosystem $(\A, G, \de)$.
Suppose that $\psi \colon \cenv(\A, G, \de) \to B$ is a $*$-homomorphism that is completely isometric on $\A$.
Then $\psi$ is faithful on the fixed point algebra $\left[ \cenv(\A,G,\de) \right]_e$.
\end{corollary}

\begin{proof}
Without loss of generality assume that $\psi$ is surjective. Let $\iota\colon \A\rightarrow\cenv(\A)$ be the natural inclusion. 
Since $\psi$ is surjective and completely isometric on $\A$, the pair $(B, \psi(\iota\otimes \id) \de)$ is a C*-cover for $\A$. 
By the defining property of $\cenv(\A)$, there exists a surjective $*$-homomorphism $\phi\colon B\rightarrow \cenv(\A)$ so that $\phi  \big(\psi(\iota\otimes \id)\de \big)= \iota$. 
Therefore
\begin{equation} \label{eq;ceident}
(\phi\psi)(\iota(a) \otimes u_g) =  \iota(a), \mbox{ for all } a \in \A_g, g \in G.
\end{equation}

Now for our purposes, it suffices to show that $\phi\circ \psi$ is injective on $\left[ \cenv(\A, G, \de) \right]_e$.  
Notice that
\[
\spn\left\{ \prod\iota(a_{g_i})\iota(a_{h_i})^*\otimes u_{g_{i}h_{i}^{-1}} \mid a_{g_i}\in \A_{g_i} ,a_{h_i} \in \A_{h_i} \right\}
\]
is a dense subset of  $\cenv(\A, G, \de)$ with
\[
\spn\left\{ \prod\iota(a_{g_i})\iota(a_{h_i})^*\otimes u_{g_{i}h_{i}^{-1}} \mid a_{g_i}\in \A_{g_i} ,a_{h_i} \in \A_{h_i} ,  \prod g_ih_i^{-1} = g \right\} \subseteq \left[ \cenv(\A, G, \de)\right]_g, \,\, g \in G.
\]
Therefore, (the last two sentences of) Remark~\ref{r;Qui} implies that
\[
\left[ \cenv(\A, G, \de)\right]_e=\overline{\spn}\left\{ \prod\iota(a_{g_i})\iota(a_{h_i})^*\otimes1\mid a_{g_i}\in \A_{g_i} ,a_{h_i} \in \A_{h_i},  \prod g_ih_i^{-1} = e \right\}.
\]
Now (\ref{eq;ceident}) implies that $\phi \psi$ is the inverse of the ampliation map on $\left[ \cenv(\A, G, \de)\right]_e$ and the conclusion follows.
\end{proof}

Let us close this section with a discussion on gradings of C*-algebras in the sense of \cite{Exe17}.
\begin{definition}
Let $B$ be a C*-algebra and $G$ a discrete group. A collection of closed linear subspaces $\{B_g\}_{g\in G}$ of $B$ is called a \emph{grading} of $B$ by $G$ if
\begin{enumerate}

\item $B_g B_h \subseteq B_{gh}$

\item $B_g^* = B_{g^{-1}}$

\item $\sum_{g\in G} B_g$ is dense in $B$.

\end{enumerate}
If in addition there is a conditional expectation $E : B \rightarrow B_e$ which vanishes on $B_g$ for $g\neq e$, we say that the pair $(\{B_g\}_{g \in G}, E)$ is a \emph{topological} grading of $B$.
\end{definition}

When $\delta$ is a coaction on a C*-algebra $B$, the spectral subspaces $B_g$ for $g\in G$ comprise a topological grading for $B$ with conditional expectation $E_e = (\id \otimes F_e) \circ \de$ where $F_e : \ca(G) \rightarrow B$ is the $e$-th Fourier coefficent. Completely contractive maps $E_g : B \rightarrow B_g$ can be similarly defined by setting $E_g:= (\id \otimes F_g) \circ \de$, where $F_g : C^*(G) \rightarrow \bC $ is the $g$-th Fourier coefficient.

A grading of a C*-algebra by a group constitutes a Fell bundle over the group, and every Fell bundle arises this way, but not uniquely. Indeed, there may be many non-isomorphic graded C*-algebras whose gradings are all equal to a pre-assigned Fell bundle $\B$. At one extreme sits the maximal C*-algebra $C^*(\B)$, which is universal  for representations of $\B$, while at the other extreme is the minimal (reduced) cross sectional algebra $C^*_\lambda(\B)$ which is defined via the left regular representation  of $\B$ on $\ell^2(\B)$.  We refer to \cite{Exe97, Exe17} for the precise definitions and details.

%%%%%%%%%%%%%%%%%%%%%%%%%%%%%%%%
\begin{definition}
Let $\B = \{B_g\}_{g \in G}$ be a topological grading for a C*-algebra $B$ over a group $G$. We say that an ideal  $I \lhd B$ is \emph{induced} if $I = \sca{I \cap B_e}$.
\end{definition}

If $\de \colon B \to B \otimes \ca(G)$ is a coaction on a C*-algebra and $I \lhd B$ is an induced ideal then $\de$ induces a faithful coaction $B /I$, see for example \cite[Proposition A.1]{CLSV11}. Normal actions also descend through induced ideals when $G$ is exact, see for example \cite[Proposition A.5]{CLSV11}. 

%%%%%%%%%%%%%%%%%%%%%%%%%%%%%%%%
\section{C*-envelopes of cosystems as co-universal C*-algebras } \label{S;main}
%%%%%%%%%%%%%%%%%%%%%%%%%%%%%%%%

In this section we consider the cosystem consisting of the Fock tensor algebra $\T_\la(X)^+$ of a compactly aligned product system $X$ over  a right LCM subsemigroup $P$  of a group $G$, together with a natural  coaction. We prove that   the associated C*-envelope has the co-universal property of \cite[Theorem 4.1]{CLSV11} with respect to $X$.

Let $\wt{t} =\{\wt{t}_p\}_{p \in P}$ be the \textit{universal Toeplitz representation for $X$}.
By the universal property of $\T(X)$ there is a canonical $*$-homomorphism
\[
\wt{\de} \colon \T(X) \longrightarrow \T(X) \otimes \ca(G) : \wt{t}_p(\xi_p) \longmapsto \wt{t}_p(\xi_p) \otimes u_p.
\]
Sehnem \cite[Lemma 2.2]{Seh18} has shown that $\wt{\de}$ is a non-degenerate and faithful coaction of $\T(X)$, where each spectral subspace $\T(X)_g$ with $g \in G$ is the closed linear span of the products
\[
\wt{t}_{p_1}(\xi_{p_1}) \wt{t}_{p_2}(\xi_{p_2})^* \wt{t}_{p_3}(\xi_{p_3}) \cdots  \wt{t}_{p_n}(\xi_{p_n})^* \qfor p_1 p_2^{-1}p_3 \cdots p_n^{-1} = g \ \ \text{and} \ \ \xi_{p_i} \in X_{p_i}
\]

Let $\wh{t}=\{\wh{t}_p\}_{p \in P}$ be the \textit{universal Nica-Toeplitz representation of $X$}. As $\N\T(X)$ is a quotient of $\T(X)$ by an induced ideal, by \cite[Proposition A.1]{CLSV11} the non-degenerate and faithful coaction of $G$ on $\T(X)$ descends canonically to one on $\N\T(X)$. Therefore, the canonical $*$-homomorphism 
\[
\wh{\de} \colon \N\T(X) \longrightarrow \N\T(X) \otimes \ca(G) : \wh{t}_p(\xi_p) \longmapsto \wh{t}_p(\xi_p) \otimes u_p
\]
defines a coaction on $\N\T(X)$.

The following proposition shows that the Fock algebra, being a reduced type object, admits a normal coaction. We state and prove the result in complete generality for future reference.

%%%%%%%%%%%%%%%%%%%%%%%%%%%%%%%%
\begin{proposition}\label{P:f coaction}
Let $P$ be a unital subsemigroup of a group $G$ and $X$ a product system over $P$ with coefficients in $A$.
Let $\ol{t}$ be the Fock representation. Then there is a normal coaction
\[
\overline{\de}\colon \T_\la(X) \longrightarrow \T_{\la}(X) \otimes \ca(G) : \ol{t}_p(\xi_p) \longmapsto \ol{t}_p(\xi_p) \otimes u_{p}.
\]
Moreover each spectral space $\T_\la(X)_g$ with $g \in G$ is generated by the products of the form
\[
\ol{t}_{p_1}(\xi_{p_1}) \ol{t}_{p_2}(\xi_{p_2})^*\ol{t}_{p_3}(\xi_{p_3})  \cdots \ol{t}_{p_n}(\xi_{p_n})^*, \quad p_1 p_2^{-1} p_3\cdots p_n^{-1} = g.
\]
\end{proposition}

\begin{proof}
Let the operator $U \colon \F X \otimes \ell^2(G) \to \F X \otimes \ell^2(G)$ be given by
\[
U (\xi_r \otimes e_g) = \xi_r \otimes e_{r g}
\foral
r \in P, g \in G.
\]
We have that $U$ is a unitary in $\L( \F X \otimes \ell^2(G) )$ for the Hilbert bimodule $\F X \otimes \ell^2(G)$ over $A \otimes \bC = A$, with
\[
U^*( \xi_r \otimes e_h) = \xi_r \otimes e_{r^{-1} g}.
\]
We can then directly verify that $U(\ol{t}(\xi_p) \otimes I) = (\ol{t}(\xi_p) \otimes \la_p) U$ for all $p \in P$.
Thus $\ad_{U}$ implements a faithful $*$-homomorphism $\de_{\la} : \T_{\la}(X) \rightarrow \T_{\la}(X) \otimes C^*_{\la}(G)$ given by
\[
\xymatrix{
\T_\la(X) \simeq \ca(\ol{t}_p(\xi_p) \otimes I \mid p \in P) \ar[rr]^{\phantom{oooooo} \ad_{U}}
& & 
\ca(\ol{t}_p(\xi_p) \otimes \la_p \mid p \in P).
}
\]

Let $\ol{t}_* : \T(X) \rightarrow \T_{\la}(X)$ be the canonical surjection induced by $\ol{t}$. Since the spectral subspaces $\T(X)_g$, $g \in G$, for the coaction $\wt{\de}$ are the closed linear span of products
\[
\wt{t}_{p_1}(\xi_{p_1}) \wt{t}_{p_2}(\xi_{p_2})^* \wt{t}_{p_3}(\xi_{p_3}) \cdots  \wt{t}_{p_n}(\xi_{p_n})^* \qfor p_1 p_2^{-1}p_3 \cdots p_n^{-1} = g.
\]
and $\sum_{g\in G} \T(X)_g$ is dense in $\T(X)$, we see that the same persists after the application of $\ol{t}_*$. More precisely, we let $\T_{\la}(X)_g$ for $g \in G$ be the subspaces given by the closed linear span of
\[
\ol{t}_{p_1}(\xi_{p_1}) \ol{t}_{p_2}(\xi_{p_2})^* \ol{t}_{p_3}(\xi_{p_3}) \cdots  \ol{t}_{p_n}(\xi_{p_n})^* \qfor p_1 p_2^{-1}p_3 \cdots p_n^{-1} = g.
\]
As $\ol{t}_*$ is surjective, we get that $\sum_{g\in G} \T_{\la}(X)_g$ is dense in $\T_{\la}(X)$, and that $\ol{t}_*(\T(X)_g) = \T_{\la}(X)_g$. 

Since by definition $\de_{\la}(a) = a \otimes \la_g$ for $a\in \T_{\la}(X)_g$, we get that $\de_{\la}$ satisfies the conditions of Proposition~\ref{P:Fell ind}. Hence, there is a normal coaction $\de$ of $G$ on $\T_{\la}(X)$ whose spectral subspaces are $\T_{\la}(X)_g$.
\end{proof}

The faithful conditional expectation $\ol{E}:= E_{\ol{\de}}  \colon \T_\la(X) \rightarrow  \T_\la(X)_e$ given by the normal coaction $\ol{\de}$ of Proposition~\ref{P:f coaction} satisfies
\[
\ol{E}(f) = \sum_{r \in P} Q_r f Q_r
\foral
f \in \T_\la(X),
\]
for the projections $Q_{r} \colon \F(X) \to X_r$. Indeed, the above sum of compressions to the $(r,r)$-entries in $\L(\F X)$ acts as the identity on $ \T_\la(X)_e$ and zeroes all other spectral subspaces of $\ol{\de}$. By Remark~\ref{r;Qui}, $\ol{E}$ demonstrates exactly the same behavior on the spectral subspaces of $\ol{\de}$ and since these form a grading of $\T_\la(X)$, the conclusion follows.

We need the following auxiliary proposition. Even though it can be deduced from \cite[Theorem 2.17]{KL19b}, we give instead a self-contained proof.
%%%%%%%%%%%%%%%%%%%%%%%%%%%%%%%%
\begin{lemma}  \label{P:1-1 Fock cexp}
Let $P$ be a right LCM subsemigroup of a group $G$ and $X$ be a compactly aligned product system over $P$.
If $\Phi \colon \N\T(X) \to \T_\la(X)$ is the canonical $*$-epimorphism, then $\Phi$ is faithful on $\N\T(X)_e$.
\end{lemma}

\begin{proof}
Let $\N\T(X) = \ca(\wh{t})$ and $\T_\la(X) = \ca(\ol{t})$.
Since $\N\T(X)_e = B_{P, \wh{t}}$, Proposition~\ref{P:purified} implies that it suffices to prove the injectivity of $\Phi$ on every $B_{F, \wh{t}}$, where $F$ ranges over all finite and {$\vee$-closed}  subsets of $P$.
Towards this end suppose that $k_p \in \K X_{p}$ for $p\in F$ and
\[
0 \neq f := \sum_{p \in F} \wh{t}_p(k_p) \in \ker \Phi.
\]
Let $p_0$ be minimal in $F$ such that $\wh{t}_{p_0}(k_{p_0}) \neq 0$; then $k_{p_0} \neq 0$.  Recall that $\ol{t}_p(k_{p})  \xi_q \neq  0$ for $\xi_q\in X_q$ implies $p\leq q$,
so by minimality of $p_0$ in $F$ we have that  $Q_{p_0}\ol{t}_p(k_{p})  Q_{p_0} =0$ for $p\in F\setminus \{p_0\}$. Hence
\[
k_{p_0} = Q_{p_0} \left(\sum_{p \in F} \ol{t}_p(k_{p}) \right) Q_{p_0} = Q_{p_0} \Phi(f) Q_{p_0} = 0,
\]
which is a contradiction.
\end{proof}

%%%%%%%%%%%%%%%%%%%%%%%%%%%%%%%%

%%%%%%%%%%%%%%%%%%%%%%%%%%%%%%%%
\begin{proposition}\label{T:Fock is Fell}
Let $P$ be a right LCM subsemigroup of a group $G$ and let  $X$ be a compactly aligned product system over $P$.
Consider the Fell bundle
\[
\N X := \{ [\N\T(X)]_g \}_{g \in G}
\]
induced by the canonical coaction $\wh{\de}$ of $G$ on $\N\T(X)$.
Then
\[
\ca(\N X) \simeq \N\T(X)
\qand
\ca_\la(\N X) \simeq \T_\la(X),
\]
i.e., $\N\T(X)$ is the full C*-algebra of the bundle $\N X$ and $\T_\la(X)$ is the reduced C*-algebra of the bundle $\N X$.
\end{proposition}

\begin{proof}
For the first part, let $t' : X \rightarrow B(\H)$ be a Nica-covariant representation such that the  representation 
$t'_*: \N\T(X) \to \ca(t')$ is injective. Let $\A$ be a graded C*-algebra with grading isomorphism $\phi_g : [\N\T(X)]_g \rightarrow \A_g$ for each $g\in G$. Define a map $t : X \rightarrow \A$ by setting $t_p(\xi) = \phi_p(t'_p(\xi))$ for $\xi \in X_p$ and $p \in P$. 

We claim that  $t : X \rightarrow \A$ is a Nica-covariant representation.
Indeed, if $\xi, \zeta\in X_p$, $p \in P$, then 
\begin{align*}
t_p(\xi )^* t_p( \zeta) &= \phi_{p^{-1}}(t'_p(\xi)^*)\phi_p(t'_p(\zeta)) =\phi_e\big( t'_p(\xi)^*t'_p(\zeta)\big) \\
				&=\phi_e\big(t'_e(\langle  \xi \mid  \zeta\rangle ) \big) \\
				&=t_e(\langle  \xi \mid  \zeta\rangle)
				\end{align*}
and so $t : X \rightarrow \A$ is a Toeplitz representation. Similar arguments show that $t_p(k_p)=\phi_e(t'_{p}(k_p))$, for any rank-one operator $k_p \in \K X_p$, and thus by continuity for any compact operator operator $k_p \in \K X_p$. From this it is clear that the Nica-covariance of $t'$ implies that of $t$.

Having established the Nica-covariance of $t$, we now have an induced $*$-surjection $t_* : \N\T(X) \rightarrow \A$ such that $\phi:=t_*$ restricts to $\phi_g$ on $\N\T(X)_g$. This shows that $\ca(\N X) \simeq \N\T(X)$.

For the second part, there exists a canonical $*$-epimorphism $\Phi \colon \N\T(X) \to \T_\la(X)$ which by definition intertwines the coactions, and thus the conditional expectations.
By Lemma~\ref{P:1-1 Fock cexp} the map $\Phi$ is faithful on the fixed point algebra of $\N\T(X)$ and thus induces an isomorphism of the Fell bundle $\N X$.
Since the conditional expectation on $\T_\la(X)$ is faithful it follows by \cite[Proposition 3.7]{Exe97} that $\T_\la(X)$ is $*$-isomorphic to the reduced C*-algebra of the Fell bundle $\N X$.
\end{proof}

%%%%%%%%%%%%%%%%%%%%%%%%%%%%%%%%
Our next result shows that up to a canonical $*$-isomorphism, the tensoring of any injective Nica-covariant representation of the product system $X$ with the left regular representation of $P$ produces the C*-algebra of the Fock representation of $X$.

\begin{proposition}\label{P:P coa B}
Let $P$ be a right LCM subsemigroup of a group $G$, let $X$ be a compactly aligned product system over $P$ and let $t = \{t_p\}_{p \in P}$ be an injective Nica-covariant representation of $X$.
Then there exists a faithful $*$-homomorphism
\[
\T_\la(X) \longrightarrow \ca (t) \otimes \ca_\la(P) : \ol{t}_p(\xi_p) \longmapsto t_p(\xi_p)\otimes V_p.
\]
\end{proposition}

\begin{proof}
Consider the Nica-covariant representation
\[
t\otimes V \colon X\longrightarrow C^*(t)\otimes C^*_{\la}(P): \xi_p\longmapsto t_p(\xi_p) \otimes V_p, \,\, \xi_p \in X_p, p \in P.
\]
We claim that the induced representation $(t \otimes V)_*: \N\T(X)\rightarrow \ca(t)\otimes \ca_{\la}(P)$ is faithful on the fixed point algebra $\N\T(X)_e$. According to Proposition~\ref{P:purified}, it suffices to verify injectivity on $B_{F, \wh{t}}\subseteq \N\T(X)$, where $F$ is an arbitrary finite {$\vee$-closed} subset of $P$.
Towards this end let $k_p \in \K X_p$ with $p \in F$ such that
\[
f := \sum_{p \in F} \wh{t}_p(k_p) \neq 0
\qand 
(t \otimes V)_*(f)=\sum_{p \in F} t_p(k_p) \otimes V_p V_p^* = 0.
\]
Let $p_0 \in F$ be minimal such that $\wh{t}_{p_0}(k_{p_0}) \neq 0$; then $k_{p_0} \neq 0$.
We directly compute
\[
t_{p_0}(k_{p_0}) 
= (I \otimes P_{\bC e_{p_0}}) \bigg( \sum_{p \in F}  t_p(k_p) \otimes V_p V_p^* \bigg) (I \otimes P_{\bC e_{p_0}}) 
= 0.
\]
Since $t$ is injective, we have that $k_{p_0} = 0$, a contradiction that establishes the desired injectivity for $(t \otimes V)_*$ on $\N\T(X)_e$.

It follows now that $(t\otimes V)_*$ is injective on each fiber of the bundle $\N X$ of Proposition~\ref{T:Fock is Fell} and so $C^*(t \otimes V)$ becomes a cross sectional algebra of $\N X$. According to Proposition~\ref{T:Fock is Fell}, $\T_\la(X)$ is the reduced cross sectional algebra of $\N X$ and so \cite[Theorem 3.3]{Exe97} implies a map
\begin{equation} \label{finally}
\ca (t) \otimes \ca_\la(P)  \longrightarrow \T_\la(X)  : t_p(\xi_p)\otimes V_p \longmapsto \ol{t}_p(\xi_p) .
\end{equation}
The canonical expectation of $C^*(t \otimes V)$ onto $(t\otimes V)_*(\N\T(X)_e)$ coincides with $\id \otimes E_P$, where $E_P$ is compression on the diagonal, and so it is faithful. By \cite[Proposition 3.7]{Exe97}, the map in (\ref{finally}) is faithful and the conclusion follows.
\end{proof}

As a consequence the injective representations of a product system $X$ that inherit the coaction of $G$ produce C*-covers for $(\T_\la(X)^+, G, \dep)$. 
%%%%%%%%%%%%%%%%%%%%%%%%%%%%%%%%
\begin{proposition}\label{P:P cover}
Let $P$ be a right LCM subsemigroup of a group $G$ and let $X$ be a compactly aligned product system over $P$. 
Let $(B, G, \de)$ be a cosystem for which there exists an equivariant epimorphism
\[
\phi \colon \T_\la(X) \longrightarrow B.
\]
If $\phi|_{\ol{t}(X_e)}$ is faithful, then $\phi|_{\T_\la(X)^+}$ is completely isometric and therefore $(B, G, \de)$ forms a C*-cover for $(\T_\la(X)^+, G, \dep)$.
\end{proposition}

\begin{proof}
Consider the unital completely positive map
\[
\psi \colon \ca(G) \longrightarrow \ca_\la(G) \longrightarrow \B(\ell^2(P)) : u_g \mapsto \la_g \mapsto P_{\ell^2(P)} \la_g |_{\ell^2(P)}
\]
which is multiplicative on the subalgebra of $\ca(G)$ generated by all $u_p$, $p \in P$. Since $\psi(u_p) = V_p$ for all $p\in P$, the following diagram of completely contractive homomorphisms
\[
\xymatrix{
\T_\la(X)^+ \ar[d]_{\phi} \ar[rr] & & B \otimes \ca_\la(P) \\
B \ar[rr] & & B \otimes \ca(G) \ar[u]_{\id \otimes \psi}
}
\]
commutes.
By Proposition \ref{P:P coa B} the upper horizontal map is a restriction of a faithful $*$-homomor\-phism and thus it is completely isometric.
Hence $\phi$ is completely isometric.
\end{proof}

\begin{definition}
Following \cite[Section 4]{CLSV11} we say that a representation $t$ of a product system $X$ is  \emph{gauge-compatible}, or simply \emph{equivariant} if $\ca(t)$ admits a coaction of $G$  that makes the canonical epimorphism $\T(X)\rightarrow \ca(t)$ equivariant with respect to the natural  (gauge) coaction of $G$ on $\T(X)$. 
\end{definition}
Carlsen, Larsen, Sims and Vittadello  proposed the idea of a co-universal C*-algebra with respect to gauge-compatible, injective,  Nica-covariant  representations of $X$. Roughly speaking, such a co-universal C*-algebra is  the smallest C*-algebra carrying a coaction of $G$ that is generated by an  equivariant,  injective,  Nica-covariant representation of $X$. For the precise formulation see Definition \ref{D:couniversal} below. Carlsen, Larsen, Sims and Vittadello went on to prove that under various hypotheses on the product system, the reduced cross sectional algebra of the  Fell bundle associated with $\N\O(X)$ does satisfy the co-universal property \cite[Theorem 4.1]{CLSV11}.    

%%%%%%%%%%%%%%%%%%%%%%%%%%%%%%%%
\begin{definition}\label{D:couniversal}
 Let $P$ be a right LCM subsemigroup of a group $G$ and  let $X$ be a compactly aligned product system over $P$.
Suppose  $(\couni, G,\gamma)$ is a cosystem and $j:X \to \couni$ is a Nica-covariant isometric representation, with integrated version denoted by $j_*: \N\T(X) \to \couni$.
We say that  $(\couni, G, \gamma, j)$ has the \emph{co-universal property for equivariant, injective, Nica-covariant  representations of $X$} if
\begin{enumerate}
\item 
$j_e$ is faithful;
\item $j_*: \N\T(X) \to \couni$ is  $\hat{\de}$-$\gamma$ equivariant; and 
\item for every  equivariant, injective,  Nica-covariant representation $t: X \rightarrow \ca(t)$, 
there is a surjective $*$-homomorphism $\phi : \ca (t) \rightarrow \couni$ such that $$\phi t_p(\xi_p) = j_p(\xi_p), \mbox{ for all } \xi_p \in X_p \mbox{ and } p \in P.$$
\end{enumerate}
Notice that, as observed at the beginning of \cite[Section 4]{CLSV11},  the map $\phi$ is automatically equivariant because $j_*$ and $t_*$ are surjective.
\end{definition}
\begin{remark}
We have eschewed the notation $\N\O^r(X)$  used in \cite{CLSV11}  because there is a certain degree of ambiguity in relation to its meaning.  On the one hand, it is clear from \cite[Introduction]{CLSV11} and the statement of  \cite[Theorem 4.1]{CLSV11} 
that $\N\O^r(X)$ is implicitly intended to mean any C*-algebra that satisfies the co-universal property, 
while on the other hand at the start of the proof of \cite[Theorem 4.1]{CLSV11}, $\N\O^r(X)$ is explicitly defined to be the reduced  cross sectional algebra of the Fell bundle 
$\N\O X:= \{ [\N\O(X)]_g \}_{g \in G}$ of  
the natural coaction of $G$ on $\N\O(X)$. This causes no problem so long as the product system $X$ satisfies the assumptions of \cite[Theorem 4.1]{CLSV11}, but there is a definite clash for some examples that do not satisfy those hypothesis, see e.g. \cite[Remark 4.2]{CLSV11}. We point out that there is no ambiguity in \cite{DK20} where the notation is exclusively used to denote any C*-algebra that satisfies the co-universal property.
\end{remark}

Our next result shows that the C*-envelope of the tensor algebra $\T_\la(X)^+$ taken with its  natural coaction satisfies the co-universal property, thus establishing the existence of  a co-universal object for general compactly aligned product systems over  right LCM semigroups. This completes the program initiated in \cite{CLSV11} and continued in \cite{DK20}.

%%%%%%%%%%%%%%%%%%%%%%%%%%%%%%%%

\begin{theorem} \label{T:co-univ}
Let $P$ be a right LCM subsemigroup of a group $G$ and let $X$ be a compactly aligned product system over $P$.
Let 
$\dep \colon \T_\la(X)^+ \rightarrow \T_\la(X)^+ \otimes \ca(G)$ be
the restriction of the coaction from Proposition \ref{P:f coaction} to $\T_\la(X)^+$. Then the C*-envelope  
$
 (\cenv(\T_\la(X)^+, G, {\dep}),\delenv,\iotenv)
$
of the cosystem $(\T_\la(X)^+, G, {\dep})$ 
 satisfies the co-universal property associated  with $X$.
In particular, the canonical coaction on the co-universal object is  normal.
\end{theorem}

\begin{proof}
By definition $\cenv(\T_\la(X)^+, G,\dep)$ is generated by an injective Nica-covariant, $G$-compatible representation of $X$.
It remains to show that it has the required co-universal property.

Let $\ol{E}$ be the faithful conditional expectation on $\T_\la(X)$ and let $\Phi \colon \N\T(X) \to \T_\la(X)$ be the canonical $*$-epimorphism.
Then we have that
\[
\ker \Phi  = \{f \in \N\T(X) \mid \ol{E} \Phi(f^*f) = 0\}.
\]
In particular, since $\Phi$ is faithful on the fixed point algebra by Lemma~\ref{P:1-1 Fock cexp}, we get that 
$$\wh{E} := (\Phi|_{[\N\T(X)]_e})^{-1} \ol{E} \Phi$$ 
is the conditional expectation on $\N\T(X)$.

Let $t$ be  an injective, Nica-covariant, equivariant representation  of $X$.
Then $\ca(t)$ admits a $G$-grading and let us write $\B = \{B_g\}_{g \in G}$ for this grading of $\ca(t)$.
Due to the existence of the conditional expectation on $\ca(t)$, by \cite[Theorem 3.3]{Exe97}  there exists a canonical equivariant $*$-epimorphism
\[
\phi \colon \N\T(X) \longrightarrow \ca(t) \longrightarrow \ca_\la(\B)
\]
where $\ca_\la(\B)$ is the reduced cross sectional  C*-algebra of the Fell bundle $\B$. Let us write $E'$ for the associated faithful conditional expectation on $\ca_\la(\B)$.
For $f \in \ker \Phi$ we have that 
$$\wh{E}(f^*f) = (\Phi|_{[\N\T(X)]_e})^{-1} \ol{E} \Phi(f^*f) = 0.$$
As $\phi$ intertwines the conditional expectations $\wh{E}$ and $E'$, we derive that $E'(\phi(f^*f)) = 0$ and so $\phi(f) = 0$, because $E'$ is faithful.
Since $f$ was arbitrary in $\ker \Phi$ we get that $\phi(\ker \Phi) = \{0\}$.
Hence there is an induced $*$-homomorphism $\phi'$ that makes the following diagram
\[
\xymatrix{
\N\T(X) \ar[rr]^{\phi} \ar[rd]^{\Phi} & & \ca_\la(\B) \\
& \T_\la(X) \ar@{.>}[ur]^{\phi'} &
}
\]
commutative.
By construction $\phi'$ is equivariant.
Since $t_e$ is faithful we have that $A \hookrightarrow B_e \subseteq \ca_\la(\B)$ faithfully.
Then, by Proposition \ref{P:P cover} we have that $\ca_\la(\B)$ is a C*-cover of $(\T_\la(X)^+, G, \dep)$.
Therefore we have the following $*$-epimorphisms
\[
\ca(t) \longrightarrow \ca_\la(\B) \longrightarrow \cenv(\T_\la(X)^+, G, \dep),
\]
which establishes that  $\cenv(\T_\la(X)^+, G, \dep)$ satisfies the  co-universal property for $X$. 
The last sentence in the statement of the theorem follows from Corollary~\ref{C:normal}.
\end{proof}

\begin{remark}\label{R:aplug4env}
Theorem \ref{T:co-univ} shows that every compactly aligned product system 
over a right LCM subsemigroup of a group
 has an associated co-universal C*-algebra. This generalizes \cite[Theorem 4.1]{CLSV11} 
by removing the assumption that $X$ is injective or that $P$ is directed and the augmented left actions are injective,
 and it also generalizes \cite[Theorem 3.3]{DK20} by removing the assumption that $P$ is abelian.  
We have been  able to do this through the use of nonselfadjoint techniques adapted to the setting of operator algebras with a coaction,
which ultimately relies on the existence of the usual C*-envelope, through our Theorem \ref{T:co-env}. 
\end{remark}

%%%%%%%%%%%%%%%%%%%%%%%%%%%%%%%%
\section{Co-universality and Sehnem's covariance algebras} \label{S;Sehnem}
%%%%%%%%%%%%%%%%%%%%%%%%%%%%%%%%
In \cite{CLSV11}, Carlsen, Larsen, Sims and Vittadello show that under certain hypothesis on a product system $X$, the reduced cross sectional algebra of the Fell bundle $\N\O X:= \{ [\N\O(X)]_g \}_{g \in G}$ of the natural coaction of $G$ on the C*-algebra $\N\O(X)$ satisfies the co-universal property. But examples such as \cite[Remark~4.2]{CLSV11} indicate that the same bundle may fail to produce a co-universal $\ca$-algebra for product systems $X$ outside the framework of \cite{CLSV11}. This raises the question of whether a different bundle might do the job. We settle this question in the present section by considering the  Fell bundle determined by the natural coaction on Sehnem's covariance algebra \cite{Seh18}.
We begin by establishing the notation and reviewing the basic details of Sehnem's construction.

Let $P$ be a unital subsemigroup of a group $G$ and let $X = \{X_p\}_{p\in P}$ be a product system over $P$ with coefficients in $A:=X_e$.
For a finite set $F \subseteq G$ let
\[
K_F := \bigcap_{g \in F} gP.
\]
For $r \in P$ and $g \in F$ define the ideal of $A$ given by
\[
I_{r^{-1} K_{\{r,g\}}} :=
\begin{cases}
\bigcap_{t \in K_{\{r,g\}}} \ker \varphi_{r^{-1}t} & \text{if } K_{\{r,g\}} \neq \mt \text{ and } r \notin K_{\{r,g\}},\\
A & \text{otherwise}.
\end{cases}
\]
Then let
\[
I_{r^{-1} (r \vee F)} := \bigcap_{g \in F} I_{r^{-1} K_{\{r,g\}}},
\]
and let the C*-correspondences
\[
X_F := \oplus_{r \in P} X_r I_{r^{-1} (r \vee F)}
\qand
X_F^+ := \oplus_{g \in G} X_{gF}.
\]
For every $p \in P$ we define
\[
t_{F, p}(\xi_p) (\eta_r) = M_{p,r}(\xi_p \otimes \eta_r) \in X_{pr} I_{(pr)^{-1}(pr \vee pF)}, \foral \eta_r \in X_r I_{r^{-1} (r \vee F)}.
\]
This is well-defined as $I_{r^{-1}(r \vee F)} = I_{(pr)^{-1}(pr \vee pF)}$ for all $r \in P$, and $I_{r^{-1}(r \vee F)} = I_{(s^{-1}r)^{-1}(s^{-1}r \vee s^{-1}F)}$ for all $r \in sP$.
Therefore we obtain a representation $t_F:=\{t_{F, p}\}_{p \in P}$ of $X$ on $\L(X_F^+)$ that integrates to a representation
\begin{equation}\label{E:repsPhi}
\Phi_F \colon \T(X) \longrightarrow \L(X_F^+).
\end{equation}
Now let the projections 
\[
Q_{g,F} \colon X_F^+ \longrightarrow X_{gF}
\]
and define 
\[
\nor{f}_F := \nor{Q_{e,F} \Phi_F(f) Q_{e,F}} \foral f \in \left[ \T(X) \right]_e.
\]
In particular we have that 
\[
t_{F,p}(\xi_p) Q_{g, F} = Q_{pg, F} t_{F,p}(\xi_p)
\qand
t_{F,p}(\xi_p)^* Q_{g,F} = Q_{p^{-1}g, F} t_{F,p}(\xi_p)^*.
\]
and so $Q_{e,F}$ is reducing for the fixed point algebra under $\Phi_F$.

%%%%%%%%%%%%%%%%%%%%%%%%%%%%%%%%
\begin{definition}\cite[Definition 3.2]{Seh18}
A Toeplitz representation is called \emph{strongly covariant} if it vanishes on the ideal $\I_e \lhd \left[ \T(X) \right]_e$ given by
\[
\I_e := \{f \in \left[ \T(X) \right]_e \mid \lim_F \nor{f}_F = 0\},
\]
where the limit is taken with respect to the partial order induced by inclusion on finite sets of $P$.
We denote by $A \times_X P$ the universal C*-algebra with respect to the strongly covariant representations of $X$.
\end{definition}

That is, $A \times_X P$ is the quotient $\T(X)/\I_\infty$ for the ideal $\I_\infty \lhd \T(X)$ generated by $\I_e$. For example, in the simplest case of a product system generated by a single correspondence, $A \times_X P$  coincides with the Cuntz-Pimsner algebra, see \cite[Section 4]{Seh18}.
One of the important points of Sehnem's theory is that $A \times_X P$ does not depend on the group $G$ where $P$ embeds, while $A \hookrightarrow A \times_X P$ faithfully.
As a quotient by an induced ideal of $\T(X)$, it follows that $A \times_X P$ inherits the  coaction of $G$ \cite[Lemma 3.4]{Seh18}.
The following is the main theorem of \cite{Seh18}.

%%%%%%%%%%%%%%%%%%%%%%%%%%%%%%%%
\begin{theorem}\cite[Theorem 3.10]{Seh18}
Let $P$ be a unital subsemigroup of a group $G$ and let $X$ be a product system over $P$ with coefficients in $A$.
Then a $*$-homomorphism of $A \times_X P$ is faithful on $A$ if and only if it is faithful on the fixed point algebra $\left[ A \times_X P \right]_e$.
\end{theorem}

%%%%%%%%%%%%%%%%%%%%%%%%%%%%%%%%
The construction of Sehnem \cite{Seh18} encompasses a number of variants that have appeared in the literature.
This applies to the case where $(G,P)$ is a weak quasi-lattice ordered pair and $X$ is a compactly aligned product system such that  $X$ is faithful, or  $P$ is directed and the representation of $X$ in $\N\O(X)$ is faithful.
In this case Sehnem \cite[Proposition 4.6]{Seh18} obtains that $A \times_X P$ is the Cuntz-Nica-Pimsner algebra $\N\O(X)$ of Sims-Yeend \cite{SY10}.
Our next theorem  confirms that the Fell bundle of the covariance algebra $A \times_X P$ provides the right setup for co-universality.
Indeed,  our equivariant C*-envelope coincides with the reduced cross sectional C*-algebra of the Fell bundle determined by the natural coaction on $A \times_X P$. 

%%%%%%%%%%%%%%%%%%%%%%%%%%%%%%%%
\begin{theorem}\label{T:co-un is Fell}
Let $P$ be a right LCM subsemigroup of a group $G$ and let  $X$ be a compactly aligned product system over $P$ with coefficients from $A$.
Consider the Fell bundle
\[
\S\C X := \{ [A \times_X P]_g \}_{g \in G}
\]
induced by the natural coaction of $G$ on $A \times_X P$.
Then the cross sectional algebra and the reduced cross sectional algebra of $\S\C X$ are isomorphic to the covariance algebra and to the C*-envelope, respectively: 
\[
\ca(\S\C X) \simeq A \times_X P \qand \ca_\la(\S\C X) \simeq  \cenv(\T_\la(X)^+, G, \dep).
\]
\end{theorem}

\begin{proof}
For the first part, one argues as in the proof of Proposition~\ref{T:Fock is Fell}, since the strong covariance relations live in the fixed point algebra.
For the second part note that by \cite[Theorem 3.10]{Seh18} we have an equivariant, injective,  Nica-covariant representation of $X$ into $A \times_X P$. Hence the co-universality of  $\cenv(\T_\la(X)^+, G, \dep)$ proved in Theorem \ref{T:co-univ}
implies the existence of a $*$-epimorphism
\[
\phi:A \times_X P \longrightarrow  \cenv(\T_\la(X)^+, G, \dep),
\]
which is injective on $A \hookrightarrow A \times_{X} P$ and maps generators to generators. By \cite[Theorem 3.10]{Seh18}, $\phi$ is injective on $[A \times_X P]_e$ and so it is injective on $\S\C X$. Therefore  $\cenv(\T_\la(X)^+, G, \dep)$ becomes a cross sectional algebra of the bundle $\S\C X$ with a conditional expectation on $\S\C X_e$. By the minimality property of the reduced cross sectional algebra \cite[Theorem 3.3]{Exe97}, there is a canonical $*$-epimorphism
\[
\cenv(\T_\la(X)^+, G, \dep) \longrightarrow \ca_\la(\S\C X).
\]
By Theorem \ref{T:co-univ}, the coaction on  $\cenv(\T_\la(X)^+, G, \dep)$ is normal.
Therefore the conditional expectation on  $\cenv(\T_\la(X)^+, G, \dep)$ is faithful and thus the above $*$-epimorphism is faithful.
\end{proof}

Let us see the form of the strong covariance relations for compactly aligned product systems over right LCM semigroups.
The following is proved by Sehnem in \cite[Proposition 4.2]{Seh18} for quasi-lattices, but the same proof passes to right LCM-semigroups as well.
Notice that we consider the restriction to $X_F$ rather than the representation on the entire $X_F^+$.

%%%%%%%%%%%%%%%%%%%%%%%%%%%%%%%%
\begin{proposition}\label{P:sc lcm}
%Let $(G,P)$ be a \tcr{{right LCM subsemigroup of a group}} and $X$ be a compactly aligned product system over $(G,P)$. 
Let $P$ be a right LCM subsemigroup of a group $G$ and let  $X$ be a compactly aligned product system over $P$. A representation $t = \{t_p\}_{p \in P}$ of $X$ is strongly covariant if and only if it is Nica-covariant and it satisfies
\[
\sum_{p \in F} t_{F, p}(k_p) |_{X_F} = 0 \Longrightarrow \sum_{p \in F} t_p(k_p) = 0
\]
for any finite $F \subseteq P$ and $k_p \in \K X_p$.
\end{proposition}

\begin{proof}
The proof is identical to that of \cite[Proposition 4.2]{Seh18} by replacing $p \vee q$ with $w$ whenever $p P \cap q P = w P$.
Note that the ideals are defined in such a way that if $F$ is a finite subset of $P$ and $r \in P$ then
\[
\sum_{p \in F} t_p(k_p) t_r(\eta_r) = \sum_{r \in pP} t_p(k_p) t_r(\eta_r) = \sum_{r \in p P} t_r( i_p^{r}(k_p) \eta_r)
\]
for all $\eta_r \in X_r \cdot I_{r^{-1}( r \vee F)}$ and every Nica-covariant representation $\{t_p\}_{p \in P}$ of $X$.
\end{proof}

%%%%%%%%%%%%%%%%%%%%%%%%%%%%%%%%

Let $P$ be a unital subsemigroup of a group $G$ and $X$ be a product system over $P$ with coefficients in $A$.
Let $q_\la \colon \T(X) \to \T_\la(X)$ be the canonical $*$-epimorphism.
Another interesting C*-algebra related to Sehnem's algebra  can be obtained from  $\T_\la(X)$ by taking 
the quotient of $\T_\la(X)$ by the ideal $q_\la(\I_\infty)$. We would like to analyze this quotient 
and discuss its relation with the cross sectional algebra $\ca_\la(\S\C X)$.
It is easy to see that the ideal $q_\la(\I_\infty)$ of  $\T_{\la}(X)$ is induced, hence  $\redcov$  inherits from  $\T_{\la}(X)$ 
the coaction of $G$ and, with it, a topological grading \cite[Proposition~23.10]{Exe17}.
By \cite[Theorem 3.3]{Exe97}, we then have an equivariant $*$-epimorphism
\[
\redcov \longrightarrow   \ca_\la(\S\C X),
\]
which is known to be an isomorphism if,  for instance, $G$ is exact.

We see that
the representations $\Phi_F$ from \eqref{E:repsPhi} used to define the strong covariance relations are  sub-representations of $\ol{\de}_\la \colon \T_\la(X) \to \T_\la(X) \otimes \ca_\la(G)$ for $\ol{\de}_\la = (\id \otimes \la) \ol{\de}$ where $\ol{\de}$ is the normal coaction on the Fock representation.
Indeed we can identify
\[
X_F^+ = \oplus_{g \in G} \oplus_{r \in P} X_r I_{r^{-1}(r \vee gF)}
\]
with a submodule of $\F X \otimes \ell^2(G)$ through the isometry given by
\[
X_{r} I_{r^{-1}(r \vee gF)} \ni \eta_r \mapsto \eta_r \otimes e_g \in X_r \otimes \ell^2(G).
\]
Recall here that $\F X \otimes \ell^2(G)$ is the exterior tensor product of two modules (seeing $\ell^2(G)$ as a module over $\bC$), and there is a faithful $*$-homomorphism
\[
\T_\la(X) \otimes \ca_\la(G) \subseteq \L(\F X) \otimes \B(\ell^2(G)) \hookrightarrow \L(\F X \otimes \ell^2(G)).
\]
We then see that
\[
t_{F,p}(\xi_p) = (\ol{t}_p(\xi_p) \otimes \la_p)|_{X_F^+} = \ol{\de}_\la(\ol{t}_p(\xi_p))|_{X_F^+}
\foral
p \in P,
\]
and likewise for their adjoints.
Thus $X_F^+$ is reducing under $\ol{\de}_\la(\T_\la(X))$.
Recall also that $X_F$ is reducing for $[\T(X)]_e$ as the range of the projection $Q_{e,F}$ and so we obtain the representation
\[
\bigoplus\limits_{ F \subseteq G \textup{ finite} } \Phi_F(\cdot)|_{X_F} \colon [\T(X)]_e \longrightarrow [\T_\la(X)]_e \longrightarrow \prod\limits_{ F \subseteq G \textup{ finite} }  \B(X_F).
\]
In particular almost by definition we have for an $f \in \T(X)$ that
\[
f \in \I_e \qiff \bigoplus\limits_{ F \subseteq G \textup{ finite} }  \Phi_F(f)|_{X_F} \in c_0(\B(X_F) \mid F \subseteq G \textup{ finite} ).
\]
By definition we then get that the following diagram 
\[
\xymatrix{
%line 1
[\T(X)]_e \ar[d] \ar[rr] & & [\T_\la(X)]_e \ar[rr] \ar[d] & & 
\prod\limits_{ F \subseteq G \textup{ finite} }  \B(X_F) \ar[d] \\
%line 2
[A \times_X P]_e \ar[rr] & & [\redcov]_e \ar[rr] & & 
\quo{\prod\limits_{ F \subseteq G \textup{ finite} }  \B(X_F)}{c_0(\B(X_F) \mid  F \subseteq G \textup{ finite}) }
}
\]
is commutative.
Consequently the $e$-graded $*$-algebraic relations in $\T_\la(X)$ induce strong covariance relations;  this is why strong covariance relations are Nica-covariant.
In particular we obtain the following corollary.

%%%%%%%%%%%%%%%%%%%%%%%%%%%%%%%%
\begin{corollary}\label{C:red seh inj}
Let $P$ be a unital subsemigroup of a group $G$ and let $X$ be a product system over $P$
with coefficients in $A$.
Then $A \hookrightarrow \redcov$.
Moreover a $*$-homomorphism of $\redcov$ is faithful on $A$ if and only if it is faithful on $\left[ \redcov \right]_e$.
\end{corollary}

\begin{proof}
The proof that $A \hookrightarrow \redcov$ follows by combining the commutative diagram
\[
\xymatrix{
%line 1
\T(X) \ar[d] \ar[rr] & & \T_\la(X) \ar[d] \\
%line 2
A \times_X P \ar[rr] & & \redcov
}
\]
of $*$-epimorphisms with the fact that $$A \cap c_0(\B(X_F) \mid F \subseteq G \textup{ finite} ) = \{0\},$$ which is the main argument in \cite[Lemma 3.6]{Seh18}.
The rest now follows by combining this with \cite[Theorem 3.10]{Seh18}.
\end{proof}

Surprisingly,  the key to injectivity is the normality of the coaction of $G$ on $\redcov$.

%%%%%%%%%%%%%%%%%%%%%%%%%%%%%%%%
\begin{corollary}\label{C:exa Seh}
Let $P$ be a right LCM subsemigroup of a group $G$ and  let $X$ be a compactly aligned product system over $P$ with coefficients from $A$.
Then the equivariant $*$-epimorphism 
\[
\redcov \longrightarrow  \ca_\la(\S\C X)
\]
is faithful if and only if the coaction of $G$ on $\redcov$ is normal.
\end{corollary}

\begin{proof}
First suppose that the coaction of $G$ on $\redcov$ is normal.
Then the equivariant $*$-epimorphism $\redcov \to \ca_\la(\S\C X)$ is faithful if and only if it is faithful on the fixed point algebra.  By  Corollary \ref{C:red seh inj}, this happens  if and only if it is faithful on $A$, which is the case because, by Theorem \ref{T:co-univ} and Theorem \ref{T:co-un is Fell}, we have
\[
A \hookrightarrow \T_\la(X)^+ \hookrightarrow \cenv(\T_\la(X)^+, G, \dep) \simeq \ca_\la(\S\C X). 
\]
Conversely suppose that the equivariant $*$-epimorphism is faithful.
By Corollary \ref{C:normal}, the coaction $\delenv$ on $\cenv(\T_\la(X)^+, G, \dep)$  is normal and  normality passes to $\redcov$ via the equivariant map.
\end{proof}

%%%%%%%%%%%%%%%%%%%%%%%%%%%%%%%%
\section{Reduced Hao-Ng isomorphisms}
%%%%%%%%%%%%%%%%%%%%%%%%%%%%%%%%

Let $\fG$ be a discrete group.
Suppose that an operator algebra $\fA$ admits a $\fG$-action $\alpha$ by completely isometric automorphisms $\alpha_{\fg}$ for $\fg \in \fG$. The $\fG$-action extends to an action on $\cenv(\fA)$ and one can form the reduced C*-crossed product $ \cenv(\fA) \rtimes_{\al, \la} \fG$.
 Following Katsoulis and Ramsey \cite[Definition 3.17]{KR16}, we define the nonselfadjoint reduced crossed product $\fA \rtimes_{\al, \la} \fG$ as the norm-closed subalgebra of $ \cenv(\fA) \rtimes_{\al, \la} \fG$ spanned by the canonical generators of $ \cenv(\fA) \rtimes_{\al, \la} \fG$ associated with $\fA$ and $\fG$. If $\fA \subseteq \B(H)$ and the $\fG$-action $\alpha$ extends to the $\ca$-algebra generated by $\fA$, then \cite[Corollary 3.16]{KR16} shows that $\fA \rtimes_{\al, \la} \fG$ is completely isometrically isomorphic to the subalgebra of $\B(H) \otimes \B(\ell^2(\fG))$ generated by 
\begin{align*}
\pi(a)&=\sot -\sum_{\fh \in \fG}\alpha_{\fh^{-1}}(a)\otimes P_{\fh}, \,\, \forall a \in \fA \\
U_{\fg}&=I\otimes u_{\fg}, \,\, \forall \fg \in \fG,
\end{align*}
where $P_{\fh}$, $\fh \in \fG$, denotes the projection on the $\fh$-coordinate of $\ell^2(\fG)$.

It is well known that the reduced crossed product of $\ca$-algebras is a covariant functor from the category of $\ca$-dynamical systems with morphisms the $\fG$-equivariant $*$-homomorphisms to the category of $\ca$-algebras. Specifically if $(C, \sigma , \fG)$ and $(D, \rho, \fG)$ are $\ca$-dynamical systems and $\phi: C\rightarrow D$ is an equivariant $*$-homomorphism, i.e., $\phi \sigma_{\fg}=\rho_{\fg}\phi$, for all  $\fg \in \fG$, then there exists a map 
\[
\phi\rtimes \id : C\rtimes_{\sigma, \lambda} \fG \longrightarrow D\rtimes_{\rho, \lambda} \fG 
\]
which acts as the identity on the copy of $\fG$ and satisfies $ \phi\rtimes \id (c) = \phi(c)$, for all $c \in C$.
Since the reduced crossed product of an operator algebra is defined as a subset of a $\ca$-algebra crossed product, we obtain that the crossed product of operator algebras is also a covariant functor in the above sense from the category of dynamical systems with morphisms the $\fG$-equivariant completely isometric homomorphisms to the category of operator algebras with morphisms the completely isometric homomorphisms.

Katsoulis \cite[Theorem 2.5]{Kat17} has shown that
\[
\cenv(\fA \rtimes_{\al, \la} \fG) \simeq \cenv(\fA) \rtimes_{\al, \la} \fG.
\]
The isomorphism also holds for C*-envelopes of cosystems whose coactions are equivariant with respect to the actions $\alpha$ and $\alpha \otimes \id$ of $\fG$ on $\fA$ and $\fA\otimes  \ca(G)$ respectively, as the next result shows
%%%%%%%%%%%%%%%%%%%%%%%%%%%%%%%%%
\begin{proposition}\label{P:cp env}
Let $(\fA, G, \de)$ be a (normal) cosystem.
Let $\fG$ be a group acting on $\fA$ by completely isometric automorphisms  $\alpha_{\fg}$ for $\fg \in \fG$,  such that
\begin{equation} \label{eq;equivfG}
\de \al_{\fg} = (\al_{\fg} \otimes \id) \de
\foral
\fg \in \fG.
\end{equation}
Then $G$ induces a {(normal resp.)} coaction $\widetilde{\de \rtimes \id}$ on $\fA \rtimes_{\al, \la} \fG$ and
\[
\cenv(\fA \rtimes_{\al, \la} \fG, G, \widetilde{\de \rtimes \id}) \simeq \cenv(\fA, G, \de) \rtimes_{\al, \la} \fG.
\]
\end{proposition}

\begin{proof}
For convenience suppose that $\fA \subseteq \cenv(\fA) \subseteq \B(H)$ and $\ca(G) \subseteq \B(K)$.
 The action $\al \otimes \id$ of $\fG$ on $\cenv(\fA) \otimes \ca(G)$ gives rise to the crossed product $[\cenv(\fA) \otimes \ca(G) ] \rtimes_{\alpha\otimes \id, \la} \fG$.
To make a distinction we write
\[
\pi(a) \in \B(H) \otimes \B(\ell^2(\fG))
\qand
\pi'(a \otimes u_g) \in \B(H) \otimes \B(K) \otimes \B(\ell^2(\fG))
\]
for all $a \in \fA_g$, $g \in G$, while we use the symbols
\[
U_{\fg} \in \B(H) \otimes \B(\ell^2(\fG))
\qand
U_{\fg}' \in \B(H) \otimes \B(K) \otimes \B(\ell^2(\fG))
\]
for the different shifts that define the crossed products
\[
\cenv(\fA) \rtimes_{\alpha,\la} \fG
\qand
[\cenv(\fA) \otimes \ca(G) ] \rtimes_{\alpha \otimes \id,\la} \fG.
\]
Up to a unitary interchanging of $K$ with $\ell^2(\fG)$, we get the $*$-isomorphism 
\[
\Phi \colon [\cenv(\fA) \otimes \ca(G) ] \rtimes_{\alpha \otimes \id,\la} \fG \longrightarrow [\cenv(\fA) \rtimes_{\alpha, \la} \fG] \otimes \ca(G)
\]
given by
\[
\Phi[\pi'(a \otimes u_g) U_{\fg}'] = (\pi(a) U_{\fg}) \otimes u_g
\]
for all $a \in \fA_g$, $g \in G$ and $\fg \in \fG$.
 Consider the completely isometric map
\[
\widetilde{\de \rtimes \id} := \Phi \circ (\de \rtimes \id) \colon \fA \rtimes_{\alpha, \la}  \fG \longrightarrow (\fA \rtimes_{\alpha, \la}  \fG) \otimes \ca(G).
\]
Note that
\begin{align*}
\widetilde{\de \rtimes \id}(\pi(a) U_{\fg}) &= \Phi\left((\de\rtimes \id)(\pi(a)U_{\fg})\right) \\
      &=\Phi(\pi'(a\otimes u_g)U_{\fg}') \\
      &=(\pi(a) U_{\fg}) \otimes u_g ,
\end{align*}
for all $a \in \fA_g$, $g \in G$ and $\fg \in \fG$.
Therefore, $\widetilde{\de \rtimes \id}$ is a completely isometric map which satisfies the coaction identity on $\fA \rtimes_{\al, \la} \fG$. 
Since $\de$ is non-degenerate, we have that $\sum_{g \in G} \fA_g$ is dense in $\fA$, so that also $\sum_{g \in G} [\fA \rtimes_{\al, \la} \fG]_g$ is dense in $\fA \rtimes_{\al, \la} \fG$. 
Hence $(\fA \rtimes_{\al, \la} \fG, G, \widetilde{\de \rtimes \id})$ is a cosystem.

When $\de$ is normal then we can deduce that $\widetilde{\de \rtimes \id}$ is also normal by working with $\ca_\la(G) \subseteq \B(\ell^2(G))$ and $\de_\la$, in place of $\ca(G) \subseteq \B(K)$ and $\de$, respectively.
This will give that the $*$-homomorphism
\[
(\id_{\fA \rtimes_{\al, \la} \fG} \otimes \la) \widetilde{\de \rtimes \id}= \widetilde{\de_\la \rtimes \id} \colon \fA \rtimes_{\al, \la} \fG \mapsto (\fA \rtimes_{\al, \la} \fG) \otimes \ca_\la(G)
\]
and thus $\de \rtimes \id$ is normal.

For the second part we will use the realization of the C*-envelope of a cosystem from Theorem \ref{T:co-env}.  Since we are assuming that $\fA \subseteq \cenv(\fA)$, the proof of Theorem \ref{T:co-env} (and in particular (\ref{eq;correction})) establishes a $*$-isomorphism
\[
\de_*: \cenv(\fA , G, \de)\longrightarrow  \cenv(\fA)\otimes \ca(G),
\]
which coincides with $\de$ on $\fA$. We also have a similar $*$-isomorphism
\[
(\widetilde{\de \rtimes \id})_*: \cenv(\fA \rtimes_{\al, \la} \fG, G,\widetilde{\de \rtimes \id}) \longrightarrow \cenv(\fA \rtimes_{\al, \la} \fG) \otimes \ca(G).
\]
Now \cite[Theorem 2.5]{Kat17} provides a $*$-isomorphism  
\[
\phi: \cenv(\fA \rtimes_{\al, \la} \fG) \longrightarrow \cenv(\fA) \rtimes_{\al, \la} \fG,
\]
and so we have the following faithful $*$-homomorphisms
\[
\xymatrix{
\cenv(\fA \rtimes_{\al, \la} \fG, G,\widetilde{\de \rtimes \id}) \ar[rr]^{(\widetilde{\de \rtimes \id})_{*}} \ar@{<.>}[dd]& & \cenv(\fA \rtimes_{\al, \la} \fG) \otimes \ca(G) \ar[d]^{\phi\otimes \id} \\
& & (\cenv(\fA) \rtimes_{\al, \la} \fG) \otimes \ca(G) \ar[d]^{\Phi^{-1}} \\
\cenv(\fA, G, \de) \rtimes_{\al, \la} \fG & & (\cenv(\fA) \otimes \ca(G) ) \rtimes_{\al \otimes \id, \la} \fG \ar[ll]^{(\de_{*})^{-1}\rtimes \id}
}
\]
whose composition establishes the desired $*$-isomorphism.
\end{proof}

Next we discuss the application of Proposition~\ref{P:cp env} to cosystems. A
 \emph{generalized gauge action of $\fG$ on $\T_\la(X)$} is an action $\al \colon \fG \to \Aut(\T_\la(X))$ that satisfies
\[
\al_{\fg}(\ol{t}_p(X_p)) = \ol{t}_p(X_p) \foral p \in P \text{ and } \fg \in \fG.
\]
%\footnote{It is not clear to me now what we meant by this.  Are we saying (as the referee suggests) that $\rho$ is a faithful representation of $\T_\la(X)$ that  plays the role of the representation $\pi$ of $\fA$ mentioned in the first paragraph of the section, which is induced from the identity representation of $ fA$, so that $\rho\times U$ is a faithful representation of $\T_\la(X) \rtimes_{\al, \la} \fG$? YES(Elias)}  
%Consider the reduced C*-crossed product $\T_\la(X) \rtimes_{\al, \la} \fG$ of  $\T_\la(X)$ by  $\fG$ 
Then $\al$ preserves $\T_\la(X)^+$ and thus we get the nonselfadjoint crossed product $\T_\la(X)^+ \rtimes_{\al, \la} \fG$. The ideal $q_\la(\I_\infty)$ of strong covariance in $\T_\la(X)$ is $\al$-invariant, so that $\al$ descends to a group action of $\fG$ on $\redcov$ which is also a generalized gauge action and gives the reduced crossed product $(\redcov) \rtimes_{\al, \la} \fG$.

Given a faithful representation $\rho \rtimes U $ of $\T_\la(X) \rtimes_{\al, \la} \fG$, where $\rho$ is a representation  of $\T_\la(X)$ and $U$ a unitary representation of $\fG$ such that 
$\rho\circ \al_\fg(\cdot) =U_\fg \rho(\cdot)U_\fg^*$, we can define a product system by using the concrete representations of $X$ and $\fG$.
To this end, for every $p \in P$ we define
\[
X_p \rtimes_{\al, \la} \fG := \ol{\spn} \{\rho(\ol{t}_p(\xi_p)) U_{\fg} \mid \xi_p \in X_p, \fg \in \fG\}.
\]
Since $\rho \al_{\fg}(f) = U_{\fg} \rho(f) U_{\fg}^*$ for all $f \in \T_\la(X)$ we can also write
\[
X_p \rtimes_{\al, \la} \fG := \ol{\spn} \{U_{\fg} \rho(\ol{t}_p(\xi_p)) \mid \xi_p \in X_p, \fg \in \fG\}.
\]
Consequently
\[
\ol{ (X_p \rtimes_{\al, \la} \fG) \cdot (X_q \rtimes_{\al, \la} \fG)} = X_{pq} \rtimes_{\al, \la} \fG,
\]
and thus the family
\[
X \rtimes_{\al, \la} \fG := \{X_p \rtimes_{\al, \la} \fG\}_{p \in P}
\]
defines a product system over $P$ with coefficients from $A \rtimes_{\al, \la} \fG$.
Furthermore we can write
\begin{align*}
\K(X_p \rtimes_{\al, \la} \fG) 
& = 
\ol{\spn} \{ \rho(\ol{t}_{p}(k_p)) U_{\fg} \mid k_p \in \K X_p, \fg \in \fG \} \\
& =
\ol{\spn} \{ U_{\fg} \rho(\ol{t}_{p}(k_p)) \mid k_p \in \K X_p, \fg \in \fG \}.
\end{align*}
As $X \rtimes_{\al, \la} \fG$ is defined concretely we get that
\[
i_{p}^{pq} \left( \rho(\ol{t}_p(k_p)) U_{\fg} \right) \rho(\ol{t}_{pq}(\xi_{pq}) U_{\fh})
=
\rho(\ol{t}_p(k_p)) U_{\fg} \rho(\ol{t}_{pq}(\xi_{pq}) U_{\fh}).
\]
Moreover we see that
\[
U_{\fg} \rho(\ol{t}_{p}(k_p)) \cdot
\rho(\ol{t}_{p}(k_p)) U_{\fh}
=
\rho( \al_{\fg} (\ol{t}_p(k_p) \ol{t}_q(k_q)) ) U_{\fg \fh}.
\]
As $\al_{\fg}$ defines automorphisms on the compact operators then a proof similar to \cite[Proposition 6.3]{DK20} (see also \cite[Proposition 3.2]{Kat20} for a shorter proof) shows that compact alignment of $X$ implies compact alignment of $X \rtimes_{\al, \la} \fG$.

\begin{proposition} \label{Prop;ToepHao}
Let $P$ be a right LCM subsemigroup of a group $G$ and let $X$ be a compactly aligned product system over $P$ with coefficients from $A$.
Let $\al \colon \fG \to \Aut(\T_\lambda(X))$ be a generalized gauge action  by a discrete group $\fG$. Then
\[
\T_\la(X) \rtimes_{\al, \la} \fG \simeq \T_\la(X \rtimes_{\al, \la} \fG)
\qand
\T_\la(X)^+ \rtimes_{\al, \la} \fG \simeq \T_\la(X \rtimes_{\al, \la} \fG)^+.
\]
\end{proposition}
\begin{proof}
Assume that $\T_\la(X) \subseteq \B(H)$ and for specificity let 
\begin{align*}
\rho(\ol{t}_p(\xi_p))&=\sot -\sum_{\fh \in \fG}\alpha_{\fh^{-1}}(\ol{t}_p(\xi_p))\otimes P_{\fh}, \,\,  \xi_p \in X_p, p \in P \\
U_{\fg}&=I\otimes u_{\fg}, \,\,  \fg \in \fG,
\end{align*}
acting on $H\otimes \ell^2(\fG)$. From Proposition~\ref{P:P coa B} we obtain a faithful representation
\[
\T_\la(X) \longrightarrow \T_\la(X) \otimes \ca_\la(P) : \ol{t}_p(\xi_p) \longmapsto \ol{t}_p(\xi_p)\otimes V_p
\]
and from this representation of $\T_\la(X)$ we produce the (faithful) induced representation of $\T_\la(X) \rtimes_{\al, \la} \fG$  on $H\otimes \ell^2(G)\otimes \ell^2(\fG)$ with  
\begin{equation} \label{eqq1} 
\T_\la(X) \rtimes_{\al, \la} \fG \ni  \rho(\ol{t}_p(\xi_p))U_{\fg}\longmapsto  \sot- \sum_{\fh \in \fG}\ol{t}_p(\alpha_{\fh^{-1}}(\xi_p)) \otimes V_p \otimes P_{\fh}u_{\fg}.
\end{equation}
On the other hand, Proposition~\ref{P:P coa B} applied to the identity representation of $X \rtimes_{\al, \la} \fG$, which is Nica covariant, gives a faithful representation of $\T_{\lambda}(X \rtimes_{\al, \la}\fG)$ with 
\begin{equation} \label{eqq2}
 \T_\la(X \rtimes_{\al, \la} \fG) \ni \ol{t}_p\left( \rho (\ol{t}_p(\xi_p))U_{\fg}\right) \longmapsto 
 \sot- \sum_{\fh \in \fG} \ol{t}_p(\alpha_{\fh^{-1}}(\xi_p)) \otimes P_{\fh}u_{\fg} \otimes V_p .
\end{equation}
The right sides of (\ref{eqq1}) and (\ref{eqq2}) are unitarily equivalent via the unitary that switches $\ell^2(P)$ and $\ell^2(\fG)$. Hence the algebras on the left of these equations are isomorphic with an isomorphism that preserves the corresponding tensor algebras.
\end{proof}
Since we are dealing with  two different  product systems at once, we write $\I_\infty(X) $ and $\I_\infty(  X \rtimes_{\al, \la} \fG )$ to distinguish the two relevant strong covariance ideals.
%%%%%%%%%%%%%%%%%%%%%%%%%%%%%%%%%
\begin{theorem}\label{T:hao ng}Let $P$ be a right LCM subsemigroup of a group $G$ and let $X$ be a compactly aligned product system over $P$ with coefficients from $A$.
Let $\al \colon \fG \to \Aut(\T_\lambda(X))$ be a generalized gauge action  by a discrete group $\fG$.
Then
\[
\ca_\la(\S\C(X \rtimes_{\al, \la} \fG)) \simeq \ca_\la(\S\C X) \rtimes_{\al, \la} \fG. 
\]
If in addition  the coaction of $G$ on $A \rtimes_{\al, \la} \fG$ is normal, which is the case e.g.\ when $G$ is exact,  then
\begin{equation}\label{E:co-haong}
 \T_\la(X \rtimes_{\al, \la} \fG) /q_\la(\I_\infty(  X \rtimes_{\al, \la} \fG ))    \simeq (\T_\la(X) /\I_\infty(X)) \rtimes_{\al, \la} \fG.
\end{equation}
\end{theorem}

\begin{proof}
For convenience set $Y := X \rtimes_{\al, \la} \fG$.
By construction we see that 
\[
 (\T_\la(X) \otimes \ca(G)) \rtimes_{\al \otimes \id, \la} \fG
\simeq^{\Psi}
(\T_\la(X) \rtimes_{\al, \la} \fG) \otimes \ca(G).
\]
In particular $\Psi$ is equivariant and by restricting $\Psi$ we get
\[
 (\T_\la(X)^+ \otimes \ca(G)) \rtimes_{\al \otimes \id, \la} \fG
\simeq
(\T_\la(X)^+ \rtimes_{\al, \la} \fG) \otimes \ca(G).
\]
By Theorem \ref{T:co-univ}, Proposition \ref{P:cp env} and Proposition~\ref{Prop;ToepHao} we have that
\begin{align*}
\cenv(\T_\la(Y)^+, G, \widetilde{\ol{\de} \rtimes \id}) 
 &\simeq
\cenv(\T_\la(X)^+ \rtimes_{\al, \la} \fG, G, \Psi(\widetilde{\ol{\de} \rtimes \id})) \\
 & \simeq
\cenv(\T_\la(X)^+, G, \dep) \rtimes_{\al, \la} \fG,
\end{align*}
and now an application of Theorem \ref{T:co-un is Fell} proves \eqref{E:co-haong}. The second part follows from Corollary \ref{C:exa Seh} and the proof is complete.
\end{proof}

%%%%%%%%%%%%%%%%%%%%%%%%%%%%%%%%
\begin{acknow}
The authors are thankful to the anonymous referee whose comments and suggestions improved the exposition in the paper.

Part of the research was carried out during the Focused Research Group   20frg248: Noncommutative Boundaries for Tensor Algebras at the Banff International Research Station.

Adam Dor-On was supported by the NSF grant DMS-1900916 and by the European Union's Horizon 2020 Marie Sklodowska-Curie grant No 839412. 

Evgenios Kakariadis acknowledges support from EPSRC as part of the programme ``Operator Algebras for Product Systems'' (EP/T02576X/1). 

Elias Katsoulis was partially supported by the NSF grant DMS-2054781.

Marcelo Laca was partially supported by NSERC Discovery Grant RGPIN-2017-04052. 

Xin Li has received funding from the European Research Council (ERC) under the European Union’s Horizon 2020 research and innovation programme (grant agreement No. 817597).
\end{acknow}

%%%%%%%%%%%%%%%%%%%%%%%%%%%%%%%%

\end{document}